\documentclass[a4paper,10pt]{amsart}
\usepackage{amsmath}
\usepackage{amssymb}
\usepackage{amsthm}
\usepackage{enumitem}
\usepackage{latexsym}
\usepackage{mdwlist}

\newtheorem{Def}{Definition}[section]
\newtheorem{Th}[Def]{Theorem}
\newtheorem{Le}[Def]{Lemma}
\newtheorem{Ex}[Def]{Example}
\newtheorem{Rem}[Def]{Remark}
\newtheorem{Co}[Def]{Corollary}
\newtheorem{Pro}[Def]{Proposition}

\begin{document}
\title[On the divisor-class group of monadic submonoids]{On the divisor-class group of monadic submonoids of rings of integer-valued polynomials}
\author{Andreas Reinhart}
\address{Institut f\"ur Mathematik und wissenschaftliches Rechnen, Karl-Franzens-Universit\"at, NAWI Graz, Heinrichstrasse 36, 8010 Graz, Austria}
\email{andreas.reinhart@uni-graz.at}
\begin{abstract} Let $R$ be a factorial domain. In this work we investigate the connections between the arithmetic of ${\rm Int}(R)$ (i.e., the ring of integer-valued polynomials over $R$) and its monadic submonoids (i.e., monoids of the form $\{g\in {\rm Int}(R)\mid g\mid_{{\rm Int}(R)} f^k$ for some $k\in\mathbb{N}_0\}$ for some nonzero $f\in {\rm Int}(R)$). Since every monadic submonoid of ${\rm Int}(R)$ is a Krull monoid it is possible to describe the arithmetic of these monoids in terms of their divisor-class group. We give an explicit description of these divisor-class groups in several situations and provide a few techniques that can be used to determine them. As an application we show that there are strong connections between ${\rm Int}(R)$ and its monadic submonoids. If $R=\mathbb{Z}$ or more generally if $R$ has sufficiently many ``nice'' atoms, then we prove that the infinitude of the elasticity and the tame degree of ${\rm Int}(R)$ can be explained by using the structure of monadic submonoids of ${\rm Int}(R)$.
\end{abstract}
\maketitle

\section{Introduction}
The class of Krull monoids is among the most well-studied classes of monoids in factorization theory (see cite{GHK}). It is known that the behavior of their factorizations only depends on their divisor-class group. On the other hand, there are many examples of atomic, completely integrally closed monoids that fail to be Krull. For instance, it is known that the ring of integer-valued polynomials ${\rm Int}(R)$ over an integral domain $R$ is a Krull domain if and only if $R$ is a Krull domain and ${\rm Int}(R)=R[X]$ (see \cite[Corollary I.3.15]{CC} and \cite[Corollary 2.7]{CGH}). (Note that if ${\rm Int}(R)$ is a Krull domain, then $R$ is a Krull domain, since $R^{\bullet}\subseteq {\rm Int}(R)^{\bullet}$ is a divisor-closed submonoid.) Recently, it was shown that the ring of integer-valued polynomials over a Krull domain satisfies a weaker property which is called monadically Krull \cite{F,R}. A monoid is called monadically Krull if all its divisor-closed submonoids generated by one element (i.e., monadic submonoids) are Krull monoids.

The purpose of this work is to investigate monadic submonoids of rings of integer-valued polynomials over factorial domains. Since these rings are monadically Krull it is possible to study the arithmetic of their monadic submonoids by using their divisor-class groups. The restriction to factorial domains (instead of Krull domains) is reasonable, since we are able to give more precise descriptions in this situation. We pursue two goals. The first goal is a thorough description of divisor-class groups of monadic submonoids of ${\rm Int}(R)$. We achieve this goal for monadic submonoids that are generated by polynomials with coefficients in $R$. The second goal is to show that the elasticity and the tame degree of certain rings of integer-valued polynomials are infinite. We present a proof that relies on the structure of divisor-class groups of monadic submonoids of ${\rm Int}(R)$.

The second goal is motivated by results in the literature that were proved in the recent past. More precisely, it is known that every nonempty finite subset of $\mathbb{N}_{\geq 2}$ is the set of lengths of some $f\in {\rm Int}(\mathbb{Z})^{\bullet}$ (see \cite{FF}). This is a property that ${\rm Int}(\mathbb{Z})$ shares with Krull monoids whose divisor-class group is infinite and where every class contains a height-one prime ideal (see \cite{K}). The question arises whether it is possible to describe this phenomenon in ${\rm Int}(\mathbb{Z})$ by using the theory of Krull monoids. So far, we were not able to solve this problem. Therefore, we want to pursue a simpler goal and prove that the infinitude of certain invariants (i.e., the elasticity and the tame degree) can be derived from the theory of Krull monoids.

\bigskip
In the next section we discuss the notation that is used in this work. We recall the definitions of saturated, divisor-closed, and monadic submonoids of a monoid, and present some of their elementary properties. We briefly discuss a few simple facts about rings of integer-valued polynomials. Another important notion that will be introduced is the image-content $d(f)$ of a nonzero integer-valued polynomial $f$. It is basically a greatest common divisor of the image of $f$ (over the base ring). This notion is of major importance in this work.

The main purpose of the third section is to study the structure of atoms and height-one prime ideals of monadic submonoids of ${\rm Int}(R)$. This is an important prerequisite concerning the investigation of the divisor-class group, since it is possible to describe the structure of the divisor-class group of a Krull monoid by using the $v$-product decompositions of principal ideals (generated by atoms) into height-one prime ideals. We will specifically investigate the subset of constant atoms of a monadic submonoid. We show that every constant atom generates a radical ideal. Moreover, we present a characterization result for monadic submonoids where every constant atom is a prime element. We give a complete description of the set of atoms of monadic submonoids of ${\rm Int}(R)$ that are generated by some $f\in R[X]^{\bullet}$. In what follows we study the set of height-one prime ideals of monadic submonoids generated by some $f\in R[X]^{\bullet}$ that do not contain any constant elements. Finally, we present a result which will enable us to determine the $v$-product decompositions of principal ideals into height-one prime ideals in many situations.

In the fourth section we present the first main result of this work. We show that the divisor-class group of a monadic submonoid of ${\rm Int}(R)$ generated by some $f\in R[X]^{\bullet}$ is torsion-free. Moreover, we present a simple formula to calculate the torsion-free rank in this case. We proceed with a few results that hold in a more general context. In particular, we prove a proposition which relates the $P$-adic exponents of $v$-ideals between a Krull monoid and a saturated submonoid. It is an analogue to a well-known theorem which connects $P$-adic exponents of ideals in a Dedekind domain to a subring that is also a Dedekind domain (see \cite{RR}). Moreover, it will be useful to determine the divisor-class group of monadic submonoids of ${\rm Int}(R)$ which are not covered by the first main theorem. We proceed by describing the set of height-one prime ideals that contain constant elements. These results complement the achievements in Section 3, and have several applications in the last section.

The fifth section is devoted to the construction of ``more involved'' examples of divisor-class groups. We provide basically two sufficient criteria which will enable us to decompose certain divisor-class groups of monadic submonoids into a direct product of divisor-class groups (up to an isomorphism). These criteria will be helpful in last section of this work.

In the last section we provide a few examples and discuss several consequences of the prior sections. Among them are a variety of counterexamples. For instance, it is shown that several characterization results that hold for monadic submonoids generated by some $f\in R[X]^{\bullet}$ no longer hold for arbitrary monadic submonoids. We give non-trivial examples of divisor-class groups that are torsion groups or torsion-free or none of the two. We prove that it is possible to find monadic submonoids of ${\rm Int}(\mathbb{Z})$ whose divisor-class group is torsion-free with prescribed rank. Finally, we present the second main result of this work. It shows that rings of integer-valued polynomials over certain factorial domains have infinite elasticity and tame degree.

\leftmargini25pt
\section{Notation and preliminaries}

All monoids in this work are commutative, cancellative monoids. Let $H$ be a monoid, and $T\subseteq H$ a submonoid. If $x,y\in H$, then we write $x\mid_H y$ if there is some $c\in H$ with $y=cx$.
\begin{itemize}
\item We say that $T\subseteq H$ is \textit{saturated} if for all $x,y\in T$ such that $x\mid_H y$ it follows that $x\mid_T y$.
\item $T\subseteq H$ is called \textit{divisor-closed} if for all $x,y\in H$ with $xy\in T$ we have $x\in T$.
\item If $E\subseteq H$, then let $[\![E]\!]_H$ denote the smallest divisor-closed submonoid of $H$ which contains $E$. If $x\in H$, then set $[\![x]\!]_H=[\![\{x\}]\!]_H$.
\item We say that $T\subseteq H$ is \textit{monadic} if $T=[\![x]\!]_H$ for some $x\in H$.
\end{itemize}

\bigskip
If $E\subseteq H$, then we write $[\![E]\!]$ instead of $[\![E]\!]_H$ if the monoid $H$ is the most obvious choice. Clearly, every monadic submonoid of $H$ is divisor-closed, and every divisor-closed submonoid of $H$ is saturated. Observe that if $x\in H$, then $[\![x]\!]=\{y\in H\mid y\mid_H x^k$ for some $k\in\mathbb{N}_0\}$. A subset $I\subseteq H$ is called an $s$-ideal of $H$ if $IH=I$. Let ${\rm spec}(H)$ be the set of all prime $s$-ideals of $H$. An $s$-ideal is called radical if it is an intersection of prime $s$-ideals. By $\mathfrak{X}(H)$ we denote the set of all height-one prime ($s$-)ideals of $H$, i.e., the set of all minimal nonempty prime $s$-ideals of $H$. By $H^{\times}$ resp. $\mathcal{A}(H)$ we denote the set of units of $H$ resp. the set of atoms of $H$. We say that $H$ is reduced if $H^{\times}=\{1\}$. If $x,y\in H$, then we say that $x$ and $y$ are associated (we denote this by $x\simeq_H y$) if $x=y\varepsilon$ for some $\varepsilon\in H^{\times}$. It is well-known that $\simeq_H$ defines an equivalence relation on $H$. By $H_{{\rm red}}=\{xH^{\times}\mid x\in H\}$ we denote the set of equivalence classes of $\simeq_H$. This set forms a monoid under the canonical multiplication. If $E\subseteq H$, then $F\subseteq E$ is called a system of representatives of $E$ if for every $x\in E$ there is a unique $y\in F$ such that $x\simeq_H y$. Let $L$ be a quotient monoid of $H$. For $X\subseteq L$, set $X^{-1}=\{z\in L\mid zX\subseteq H\}$ and $X_v=(X^{-1})^{-1}$. A subset $I\subseteq H$ is called a divisorial ideal (or $v$-ideal) of $H$ if $I_v=I$. Every divisorial ideal of $H$ is an $s$-ideal of $H$. Let $\mathcal{I}_v(H)$ denote the set of divisorial ideals of $H$. By $\mathcal{C}_v(H)$ we denote the divisor-class group (or $v$-class group) of $H$. It measures how far ($v$-invertible) $v$-ideals are from being principal ideals. A precise definition can be found in \cite[Definition 2.1.8]{GHK}. If $I\in\mathcal{I}_v(H)$, then let $[I]$ denote the class of $I$ in $\mathcal{C}_v(H)$. Note that $H$ is called a \textit{Krull monoid} if $H$ is a completely integrally closed Mori monoid (or equivalently, every $v$-ideal of $H$ is a finite $v$-product of height-one prime ideals of $H$). For a thorough introduction to Krull monoids we refer to \cite[Definition 2.3.1]{GHK}. We say that $H$ is \textit{monadically Krull} if $[\![x]\!]$ is a Krull monoid for every $x\in H$. Most of these notions can be defined analogously in the context of integral domains.

We want to recapitulate a few basic facts concerning saturated and divisor-closed submonoids of $H$. First let $T\subseteq H$ be saturated. Then $H^{\times}\cap T=T^{\times}$. If $H$ is a Krull monoid, then $T$ is a Krull monoid. Now let $T\subseteq H$ be divisor-closed. Then $T^{\times}=H^{\times}$, and $\mathcal{A}(T)=\mathcal{A}(H)\cap T$.

If $M$ is a set and $l\in\mathbb{N}$, then a finite sequence $(a_i)_{i=1}^l\in M^l$, will be denoted by $\underline{a}$ (i.e., $\underline{a}=(a_i)_{i=1}^l$).

Recall that if $R$ is an integral domain with quotient field $K$ and $X$ is an indeterminate over $K$, then ${\rm Int}(R)=\{f\in K[X]\mid f(x)\in R$ for all $x\in R\}$ is called the \textit{ring of integer-valued polynomials over $R$}. It is well-known that ${\rm Int}(R)^{\times}=R[X]^{\times}=R^{\times}$. Note that $R^{\bullet}=R\setminus\{0\}$ forms a monoid under multiplication. If we refer to a submonoid of $R$, then we always mean a submonoid of $R^{\bullet}$. Especially, if $E\subseteq R^{\bullet}$, then let $[\![E]\!]_R=[\![E]\!]_{R^{\bullet}}$. We say that $R$ is monadically Krull if $R^{\bullet}$ is monadically Krull.

\bigskip
Now let $R$ be a factorial domain, $K$ a field of quotients of $R$, $X$ an indeterminate over $K$, and $Q$ a system of representatives of $\mathcal{A}(R)$. For $T\subseteq R$, let ${\rm{GCD}}_R(T)$ be the set of all greatest common divisors of $T$ (in $R$), and let ${\rm{LCM}}_R(T)$ be the set of all least common multiples of $T$ (in $R$). If $f\in R[X]^{\bullet}$, then we say that $f$ is primitive if every greatest common divisors of all coefficients of $f$ is a unit of $R$. For convenience we also allow the units of $R$ to be primitive polynomials. If $q\in Q$, then let $\mathrm{v}_q:R\rightarrow\mathbb{N}_0\cup\{\infty\}$ denote the $q$-adic valuation on $R$. Let $d_Q:{\rm{Int}}(R)^{\bullet}\rightarrow R^{\bullet}$ be defined by $d_Q(g)=\prod_{p\in Q} p^{\min\{\mathrm{v}_p(g(c))\mid c\in R\}}$ for all $g\in {\rm{Int}}(R)^{\bullet}$. Set $d=d_Q$. Note that $d(g)\in {\rm{GCD}}_R(\{g(c)\mid c\in R\})$ and $\frac{g}{d(g)}\in {\rm{Int}}(R)$ for all $g\in {\rm{Int}}(R)^{\bullet}$.
It is straightforward to show that $d(f^k)=d(f)^k$ and $d(af)\simeq_R ad(f)$ for all $f\in {\rm Int}(R)^{\bullet}$, $k\in\mathbb{N}_0$, and $a\in R^{\bullet}$. Let $n\in\mathbb{N}$, $\underline{f}\in ({\rm{Int}}(R)^{\bullet})^n$ and $\underline{x}\in\mathbb{N}_0^n\backslash\{\underline{0}\}$. We say that $\underline{x}$ is \textit{$\underline{f}$-irreducible} if for all $\underline{y},\underline{z}\in\mathbb{N}_0^n$ such that $\underline{x}=\underline{y}+\underline{z}$ and $d(\prod_{i=1}^n f_i^{x_i})=d(\prod_{i=1}^n f_i^{y_i})d(\prod_{i=1}^n f_i^{z_i})$ it follows that $\underline{y}=\underline{0}$ or $\underline{z}=\underline{0}$. It is well-known (see \cite[Theorem 5.2]{R}) that ${\rm Int}(R)$ is monadically Krull. If $f\in {\rm Int}(R)^{\bullet}$, then we can deduce by \cite[Theorem 3.6]{R} and its proof that $\mathfrak{X}([\![f]\!])$, ${\rm spec}(H)$, and $\{u[\![f]\!]\mid u\in\mathcal{A}([\![f]\!])\}$ are finite sets. The remarks in this section will be used without citation.

\section{Atoms and height-one prime ideals}

In this section we present a few basic preparatory results about atoms and height-one prime ideals of monadic submonoids of ${\rm Int}(R)$. Many of the results in this section refer to monadic submonoids generated by some ``$f\in R[X]^{\bullet}$''. Note that this is a rather natural assumption because it is straightforward to prove that every monadic submonoid of ${\rm Int}(R)$ is contained in some monadic submonoid of ${\rm Int}(R)$ generated by some $f\in R[X]^{\bullet}$. (If $g\in {\rm Int}(R)^{\bullet}$ and $b\in R^{\bullet}$ are such that $bg\in R[X]$, then $[\![g]\!]\subseteq [\![bg]\!]$.) The purpose of the first result is to describe the set of ``constant atoms'' of monadic submonoids of ${\rm Int}(R)$. In particular, we show that the principal ideals generated by constant atoms are radical ideals. Furthermore, we prove that a height-one prime ideal contains at most one constant atom (up to associates).

\begin{Le}\label{general} Let $R$ be a factorial domain, and $f\in {\rm Int}(R)^{\bullet}$.
\begin{itemize}
\item[\textnormal{\textbf{1.}}] If $g\in [\![f]\!]$ and $u\in R$, then $u\mid_{[\![f]\!]} g$ if and only if $u\mid_R d(g)$.
\item[\textnormal{\textbf{2.}}] $\mathcal{A}([\![f]\!])\cap R=[\![f]\!]\cap\mathcal{A}(R)=\{u\in\mathcal{A}(R)\mid u\mid_R d(f)\}$.
\item[\textnormal{\textbf{3.}}] If $u\in\mathcal{A}([\![f]\!])\cap R$, then $uH$ is a radical ideal of $[\![f]\!]$.
\item[\textnormal{\textbf{4.}}] If $P\in\mathfrak{X}([\![f]\!])$ and $u,w\in P\cap\mathcal{A}(R)$, then $u\simeq_{[\![f]\!]} w$.
\end{itemize}
\end{Le}

\begin{proof} $\textbf{1.}$ ``$\Rightarrow$'': Let $u\mid_{[\![f]\!]} g$. There is some $v\in [\![f]\!]$ such that $g=uv$. Since $u\in R$ we infer that $d(g)\simeq_R ud(v)$, and thus $u\mid_R d(g)$. ``$\Leftarrow$'': Let $u\mid_R d(g)$. We have $d(g)\mid_{{\rm Int}(R)} g$, and thus $u\mid_{{\rm Int}(R)} g$. Since $g\in [\![f]\!]$ it follows that $u\mid_{[\![f]\!]} g$.

$\textbf{2.}$ First we show that $\mathcal{A}([\![f]\!])\cap R=[\![f]\!]\cap\mathcal{A}(R)$. ``$\subseteq$'': Let $u\in\mathcal{A}([\![f]\!])\cap R$. Observe that $u\not\in [\![f]\!]^{\times}=R^{\times}$. Let $x,y\in R$ be such that $u=xy$. It is clear that $x,y\in [\![f]\!]$. Therefore, $x\in [\![f]\!]^{\times}=R^{\times}$ or $y\in [\![f]\!]^{\times}=R^{\times}$. Consequently, $u\in [\![f]\!]\cap\mathcal{A}(R)$. ``$\supseteq$'': Let $u\in [\![f]\!]\cap\mathcal{A}(R)$. We have $u\not\in R^{\times}=[\![f]\!]^{\times}$. Let $x,y\in [\![f]\!]$ be such that $u=xy$. Observe that $0=\deg(u)=\deg(x)+\deg(y)$, and thus $x,y\in R$. Therefore, $x\in R^{\times}=[\![f]\!]^{\times}$ or $y\in R^{\times}=[\![f]\!]^{\times}$. We infer that $u\in\mathcal{A}([\![f]\!])\cap R$.

Next we show that $[\![f]\!]\cap\mathcal{A}(R)=\{u\in\mathcal{A}(R)\mid u\mid_R d(f)\}$. ``$\subseteq$'': Let $u\in [\![f]\!]\cap\mathcal{A}(R)$. Then $u\mid_{{\rm Int}(R)} f^k$ for some $k\in\mathbb{N}$. Consequently, $u\mid_{[\![f]\!]} f^k$. It follows by 1 that $u\mid_R d(f^k)=d(f)^k$. This implies that $u\mid_R d(f)$. ``$\supseteq$'': Let $u\in\mathcal{A}(R)$ be such that $u\mid_R d(f)$. By 1 we have $u\mid_{[\![f]\!]} f$. Therefore, $u\in [\![f]\!]\cap\mathcal{A}(R)$.

$\textbf{3.}$ Let $u\in\mathcal{A}([\![f]\!])\cap R$. Let $g\in [\![f]\!]$ and $n\in\mathbb{N}$ be such that $u\mid_{[\![f]\!]} g^n$. Then $u\mid_R d(g^n)=d(g)^n$ by 1, and thus $u\mid_R d(g)$. It follows by 1 that $u\mid_{[\![f]\!]} g$. Therefore, $uH$ is a radical ideal of $[\![f]\!]$.

$\textbf{4.}$ Assume to the contrary that there are $P\in\mathfrak{X}([\![f]\!])$ and $u,w\in P\cap\mathcal{A}(R)$ such that $u\not\simeq_{[\![f]\!]} w$. Let $K$ be a quotient field of $R$, $X$ an indeterminate over $K$ and $L$ a quotient group of $[\![f]\!]$. Let $h\in L$ be such that $uh,wh\in [\![f]\!]$. We have $h\in K[X]$, $h\in uh,wh\in {\rm Int}(R)$, and thus $uh(z),wh(z)\in R$ for all $z\in R$. Observe that $u\not\simeq_R w$. If $z\in R$, then $u\mid_R uwh(z)=wuh(z)$, and thus $u\mid_R uh(z)$, hence $h(z)\in R$. Therefore, $h\in {\rm Int}(R)\cap L=[\![f]\!]$. We infer that $u^{-1}[\![f]\!]\cap w^{-1}[\![f]\!]=[\![f]\!]$. Consequently, $[\![f]\!]=\{u,w\}_{v_{[\![f]\!]}}\subseteq P$, a contradiction.
\end{proof}

Let $R$ be a factorial domain, and $f\in {\rm Int}(R)^{\bullet}$. Then $\mathcal{A}([\![f]\!])\cap R$ is called the set of constant atoms of $[\![f]\!]$. Next, we characterize when every constant atom is a prime element.

\begin{Pro}\label{chards} Let $R$ be a factorial domain, and $f\in {\rm Int}(R)^{\bullet}$. The following are equivalent:
\begin{itemize}
\item[\textnormal{\textbf{1.}}] Every $P\in\mathfrak{X}([\![f]\!])$ such that $P\cap R\not=\emptyset$ is principal.
\item[\textnormal{\textbf{2.}}] For every $P\in\mathfrak{X}([\![f]\!])$ such that $P\cap R\not=\emptyset$ there is some $n\in\mathbb{N}$ such that $(P^n)_v$ is principal.
\item[\textnormal{\textbf{3.}}] Every constant atom of $[\![f]\!]$ is a prime element.
\item[\textnormal{\textbf{4.}}] $d(gh)=d(g)d(h)$ for all $g,h\in [\![f]\!]$.
\end{itemize}
If $\mathcal{C}_v([\![f]\!])$ is finite, then these conditions are satisfied.
\end{Pro}

\begin{proof} $\textbf{1.}\Rightarrow\textbf{2.}$: Trivial.

$\textbf{2.}\Rightarrow\textbf{3.}$: Let $u$ be a constant atom of $[\![f]\!]$. Since $[\![f]\!]$ is a Krull monoid there is some $P\in\mathfrak{X}([\![f]\!])$ such that $u\in P$. Some $v$-power of $P$ is principal, and thus there is some $x\in [\![f]\!]$ such that $P=\sqrt{x[\![f]\!]}$. There is some $k\in\mathbb{N}$ such that $x\mid_{[\![f]\!]} u^k$. Therefore, $x\in R$ and $x\mid_R u^k$. Since $R$ is factorial this implies that $x\simeq_R u^l$ for some $l\in\mathbb{N}$. Consequently, $x\simeq_{[\![f]\!]} u^l$. It follows that $P=\sqrt{u^l[\![f]\!]}=\sqrt{u[\![f]\!]}=u[\![f]\!]$, hence $u$ is a prime element of $[\![f]\!]$.

$\textbf{3.}\Rightarrow\textbf{1.}$: Let $P\in\mathfrak{X}([\![f]\!])$ be such that $P\cap R\not=\emptyset$. There is some $x\in P\cap R$ and some $u\in\mathcal{A}(R)$ such that $u\mid_R x$ and $u\in P$. Observe that $u$ is a constant atom of $[\![f]\!]$, hence $u$ is a prime element of $[\![f]\!]$. Therefore, $P=u[\![f]\!]$.

$\textbf{3.}\Rightarrow\textbf{4.}$: Let $g,h\in [\![f]\!]$. It is sufficient to show that $d(\frac{gh}{d(g)d(h)})=1$. Assume to the contrary that $d(\frac{gh}{d(g)d(h)})\not=1$. Then there is some $p\in\mathcal{A}(R)$ such that $p\mid_R d(\frac{gh}{d(g)d(h)})$. Obviously, $p$ is a constant atom of $[\![f]\!]$ and $p\mid_{[\![f]\!]}\frac{g}{d(g)}\frac{h}{d(h)}$. Therefore, $p\mid_{[\![f]\!]}\frac{g}{d(g)}$ or $p\mid_{[\![f]\!]}\frac{h}{d(h)}$. This implies that $p\mid_R d(\frac{g}{d(g)})=1$ or $p\mid_R d(\frac{h}{d(h)})=1$, a contradiction.

$\textbf{4.}\Rightarrow\textbf{3.}$: Let $u$ be a constant atom of $[\![f]\!]$ and $g,h\in [\![f]\!]$ such that $u\mid_{[\![f]\!]} gh$. Then $u\in\mathcal{A}(R)$ and $u\mid_R d(gh)=d(g)d(h)$. We infer that $u\mid_R d(g)$ or $u\mid_R d(h)$. Consequently, $u\mid_{[\![f]\!]} g$ or $u\mid_{[\![f]\!]} h$.

If $\mathcal{C}_v([\![f]\!])$ is finite, then for every $P\in\mathfrak{X}([\![f]\!])$ there is some $n\in\mathbb{N}$ such that $(P^n)_v$ is principal, hence 2 is satisfied.
\end{proof}

Now we show that elements of monadic submonoids that are generated by some $f\in R[X]^{\bullet}$ can be represented in form of special fractions. As a consequence, we provide a simple set of generators of the quotient group of $[\![f]\!]$. This type of representability will turn out to be a crucial ingredient for our first main result in Section 4.

\begin{Le}\label{represent} Let $R$ be a factorial domain, $K$ a quotient field of $R$, $X$ an indeterminate over $K$ and $f\in R[X]^{\bullet}$. For every $g\in [\![f]\!]$ there are some $a,b\in [\![f]\!]\cap R$ and some primitive $h\in [\![f]\!]\cap R[X]$ such that $h\mid_{R[X]} f^k$ for some $k\in\mathbb{N}$, ${\rm GCD}_R(a,b)=R^{\times}$ and $g=\frac{bh}{a}$.
\end{Le}

\begin{proof} There are some primitive $h\in R[X]$ and some $a,b\in R^{\bullet}$ such that ${\rm GCD}_R(a,b)=R^{\times}$ and $g=\frac{bh}{a}$. Since $g\in [\![f]\!]$ there are some $k\in\mathbb{N}$, $z\in R[X]^{\bullet}$ and $c\in R^{\bullet}$ such that $\frac{bhz}{ac}=f^k$. It follows that $bhz=f^kac$. Since $h$ is primitive we infer that $h\mid_{R[X]} f^k$, and thus $h\in [\![f]\!]$. Observe that $a\mid_R d(bh)\simeq_R bd(h)$. We infer that $a\mid_R d(h)$. Since $d(h)\in [\![f]\!]$ we have $a\in [\![f]\!]$. Moreover, $bh=ga\in [\![f]\!]$, and thus $b\in [\![f]\!]$.
\end{proof}

\begin{Le}\label{quotient} Let $R$ be a factorial domain, $K$ a quotient field of $R$, $X$ an indeterminate over $K$ and $f\in R[X]^{\bullet}$. Then the quotient group of $[\![f]\!]$ is generated by $([\![f]\!]\cap\mathcal{A}(R[X]))\cup [\![f]\!]^{\times}$.
\end{Le}

\begin{proof} It is sufficient to show that $[\![f]\!]\subseteq\langle ([\![f]\!]\cap\mathcal{A}(R[X]))\cup [\![f]\!]^{\times}\rangle$. Let $x\in [\![f]\!]$. By Lemma \ref{represent} there are some $h\in [\![f]\!]\cap R[X]$ and $a\in [\![f]\!]\cap R$ such that $x=\frac{h}{a}$. Since $R$ and $R[X]$ are factorial we infer that $h\in\langle ([\![f]\!]\cap\mathcal{A}(R[X]))\cup [\![f]\!]^{\times}\rangle$ and $a\in\langle ([\![f]\!]\cap\mathcal{A}(R))\cup [\![f]\!]^{\times}\rangle$. Therefore, $x\in\langle ([\![f]\!]\cap\mathcal{A}(R[X]))\cup [\![f]\!]^{\times}\rangle$.
\end{proof}

Next we give a complete description of the set of atoms of monadic submonoids of ${\rm Int}(R)^{\bullet}$ that are generated by some $f\in R[X]^{\bullet}$. A part of this description can be found in the proof of \cite[Theorem 5.2]{R}.

\begin{Pro}\label{charatoms} Let $R$ be a factorial domain, $a\in R^{\bullet}$, $n\in\mathbb{N}$ and $\underline{f}\in (\mathcal{A}(R[X])\setminus R)^n$ a sequence of pairwise non-associated elements of $R[X]$. Set $f=a\prod_{i=1}^n f_i$. Then $\{u[\![f]\!]\mid u\in\mathcal{A}([\![f]\!])\}=\{u[\![f]\!]\mid u\in\mathcal{A}(R),u\mid_R d(f)\}\cup\{\frac{\prod_{i=1}^n f_i^{y_i}}{d(\prod_{i=1}^n f_i^{y_i})}[\![f]\!]\mid\underline{y}\in\mathbb{N}_0^n,\underline{y}$ is $\underline{f}$-irreducible$\}$.
\end{Pro}

\begin{proof} ``$\subseteq$'': Let $u\in\mathcal{A}([\![f]\!])$. Observe that $u\in\mathcal{A}({\rm Int}(R))$. There is some $k\in\mathbb{N}$ such that $u\mid_{{\rm Int}(R)} f^k$.

Case 1. $u\in R$: Clearly, $u\in\mathcal{A}(R)$. We have $u\mid_R d(f^k)=d(f)^k$, and thus $u\mid_R d(f)$.

Case 2. $u\not\in R$: There are some primitive $t\in R[X]$ and some $b,c\in R^{\bullet}$ such that ${\rm GCD}_{R[X]}(bt,c)=R[X]^{\times}$ and $u=\frac{bt}{c}$. Obviously, $c\mid_R d(t)$, and thus $\frac{t}{d(t)},\frac{bd(t)}{c}\in {\rm Int}(R)$ and $u=\frac{t}{d(t)}\frac{bd(t)}{c}$. Therefore, $u\simeq_{{\rm Int}(R)}\frac{t}{d(t)}$. There are some $e\in R^{\bullet}$ and some $s\in R[X]$ such that $u\frac{s}{e}=f^k$. This implies that $bst=a^kce\prod_{i=1}^n f_i^k$. Therefore, $t\mid_{R[X]}\prod_{i=1}^n f_i^k$, hence $t\simeq_{R[X]}\prod_{i=1}^n f_i^{y_i}$ for some $\underline{y}\in\mathbb{N}_0^n\setminus\{\underline{0}\}$. Observe that $u\simeq_{{\rm Int}(R)}\frac{\prod_{i=1}^n f_i^{y_i}}{d(\prod_{i=1}^n f_i^{y_i})}$, and thus $u[\![f]\!]=\frac{\prod_{i=1}^n f_i^{y_i}}{d(\prod_{i=1}^n f_i^{y_i})}[\![f]\!]$. We need to show that $\underline{y}$ is $\underline{f}$-irreducible. Let $\underline{\alpha},\underline{\beta}\in\mathbb{N}_0^n$ be such that $\underline{y}=\underline{\alpha}+\underline {\beta}$ and $d(\prod_{i=1}^n f_i^{y_i})=d(\prod_{i=1}^n f_i^{\alpha_i})d(\prod_{i=1}^n f_i^{\beta_i})$. Clearly, $\frac{\prod_{i=1}^n f_i^{\alpha_i}}{d(\prod_{i=1}^n f_i^{\alpha_i})},\frac{\prod_{i=1}^n f_i^{\beta_i}}{d(\prod_{i=1}^n f_i^{\beta_i})}\in {\rm Int}(R)$ and $\frac{\prod_{i=1}^n f_i^{y_i}}{d(\prod_{i=1}^n f_i^{y_i})}=\frac{\prod_{i=1}^n f_i^{\alpha_i}}{d(\prod_{i=1}^n f_i^{\alpha_i})}\frac{\prod_{i=1}^n f_i^{\beta_i}}{d(\prod_{i=1}^n f_i^{\beta_i})}$. Consequently, $\frac{\prod_{i=1}^n f_i^{\alpha_i}}{d(\prod_{i=1}^n f_i^{\alpha_i})}\in {\rm Int}(R)^{\times}$ or $\frac{\prod_{i=1}^n f_i^{\beta_i}}{d(\prod_{i=1}^n f_i^{\beta_i})}\in {\rm Int}(R)^{\times}$. This implies that $\underline{\alpha}=\underline{0}$ or $\underline{\beta}=\underline{0}$.

``$\supseteq$'': Case 1. Let $u\in\mathcal{A}(R)$ be such that $u\mid_R d(f)$. We have $d(f)\in [\![f]\!]$, and thus $u\in [\![f]\!]$. Let $y,z\in [\![f]\!]$ be such that $u=yz$. Then $y,z\in R$, hence $y\in R^{\times}=[\![f]\!]^{\times}$ or $z\in R^{\times}=[\![f]\!]^{\times}$. Consequently, $u\in\mathcal{A}([\![f]\!])$.

Case 2. Let $\underline{y}\in\mathbb{N}_0^n$ be $\underline{f}$-irreducible. First we show that $\frac{\prod_{i=1}^n f_i^{y_i}}{d(\prod_{i=1}^n f_i^{y_i})}\in\mathcal{A}({\rm Int}(R))$. Let $y,z\in {\rm Int}(R)$ be such that $\frac{\prod_{i=1}^n f_i^{y_i}}{d(\prod_{i=1}^n f_i^{y_i})}=yz$. There are some $b,c,e,f\in R^{\bullet}$ and some primitive $g,h\in R[X]$ such that ${\rm GCD}(b,e)=R^{\times}$, ${\rm GCD}(c,f)=R^{\times}$, $y=\frac{bg}{e}$ and $z=\frac{ch}{f}$. This implies that $g\mid_{R[X]}\prod_{i=1}^n f_i^{y_i}$ and $h\mid_{R[X]}\prod_{i=1}^n f_i^{y_i}$. Consequently, there are some $\underline{v},\underline{w}\in\mathbb{N}_0^n$ such that $g\simeq_{R[X]}\prod_{i=1}^n f_i^{v_i}$ and $h\simeq_{R[X]}\prod_{i=1}^n f_i^{w_i}$. Observe that $e\mid_R d(g)$ and $f\mid_R d(h)$. Since $d(yz)=1$, it follows that $d(y)=d(z)=1$. Therefore, $\frac{e}{b}\simeq_R d(g)$ and $\frac{f}{c}\simeq_R d(h)$. We infer that $\frac{\prod_{i=1}^n f_i^{y_i}}{d(\prod_{i=1}^n f_i^{y_i})}\simeq_R\frac{\prod_{i=1}^n f_i^{v_i}}{d(\prod_{i=1}^n f_i^{v_i})}\frac{\prod_{i=1}^n f_i^{w_i}}{d(\prod_{i=1}^n f_i^{w_i})}$. This implies that $\underline{y}=\underline{v}+\underline{w}$ and $d(\prod_{i=1}^n f_i^{v_i})d(\prod_{i=1}^n f_i^{w_i})$. Consequently, $\underline{v}=\underline{0}$ or $\underline{w}=\underline{0}$, and thus $y\in {\rm Int}(R)^{\times}$ or $z\in {\rm Int}(R)^{\times}$.

It is clear that $\prod_{i=1}^n f_i^{y_i}\in [\![f]\!]$, hence $\frac{\prod_{i=1}^n f_i^{y_i}}{d(\prod_{i=1}^n f_i^{y_i})}\in [\![f]\!]$. This implies that $\frac{\prod_{i=1}^n f_i^{y_i}}{d(\prod_{i=1}^n f_i^{y_i})}\in [\![f]\!]\cap\mathcal{A}({\rm Int}(R))=\mathcal{A}([\![f]\!])$.
\end{proof}

Note that the set of atoms of an arbitrary monadic submonoid $H$ of ${\rm Int}(R)$ can be derived from Proposition \ref{charatoms} and the fact that $\mathcal{A}(H)=\mathcal{A}({\rm Int}(R))\cap H$. We proceed with an important lemma which will enable us to identify certain divisorial ideals of monadic submonoids of ${\rm Int}(R)$.

\begin{Le}\label{LCM} Let $R$ be a factorial domain, $f\in {\rm Int}(R)^{\bullet}$ and $g\in [\![f]\!]$. Then $\{0\}\not={\rm LCM}_R(\{\frac{d(gh)}{d(h)}\mid h\in [\![f]\!]\})\subseteq [\![f]\!]$.
\end{Le}

\begin{proof} Let $Q$ be a system of representatives of $\mathcal{A}(R)$ and $d=d_Q$. We show that there is some $T\subseteq R$ such that for all $h\in [\![f]\!]$, $h(y)\not=0$ for all $y\in T$ and $\min\{\mathrm{v}_p(h(x))\mid x\in R\}=\min\{\mathrm{v}_p(h(x))\mid x\in T\}$ for all $p\in Q$.
Set $T=\{x\in R\mid f(x)\not=0\}$. Let $h\in [\![f]\!]$. Then $h\mid_{{\rm Int}(R)} f^k$ for some $k\in\mathbb{N}$. Let $y\in T$. Then $f(y)\not=0$, hence $f(y)^k\not=0$, and thus $h(y)\not=0$. Let $p\in Q$. There is some $v\in R$ such that $\min\{\mathrm{v}_p(h(x))\mid x\in R\}=\mathrm{v}_p(h(v))$. It is straightforward to prove that there is some $m\in\mathbb{N}$ such that $\mathrm{v}_p(h(v+p^l))=\mathrm{v}_p(h(v))$ for all $l\in\mathbb{N}_{\geq m}$. Since $R\setminus T$ is finite, we can find some $n\in\mathbb{N}_{\geq m}$ such that $v+p^n\in T$. This implies that $\min\{\mathrm{v}_p(h(x))\mid x\in R\}=\mathrm{v}_p(h(v+p^l))$, hence $\min\{\mathrm{v}_p(h(x))\mid x\in R\}=\min\{\mathrm{v}_p(h(x))\mid x\in T\}$.

Next we prove that for every $p\in Q$ there is some $z\in\mathbb{N}_0$ such that $\mathrm{v}_p(\frac{d(gk)}{d(k)})\leq z$ for all $k\in [\![f]\!]$.

Without restriction let $f\not\in R^{\times}$. By \cite[Theorem 3.6 and Theorem 5.2]{R} there is some finite $\emptyset\not=U\subseteq\mathcal{A}([\![f]\!])$ such that $[\![f]\!]=[U\cup [\![f]\!]^{\times}]$. Let $p\in Q$. By Dickson's theorem (see \cite[Theorem 1.5.3]{GHK}) there is some finite $\emptyset\not=S\subseteq T$ such that ${\rm Min}(\{(\mathrm{v}_p(u(x)))_{u\in U}\mid x\in T\})=\{(\mathrm{v}_p(u(x)))_{u\in U}\mid x\in S\}$. We show that $\min\{\mathrm{v}_p(l(x))\mid x\in R\}=\min\{\mathrm{v}_p(l(x))\mid x\in S\}$ for all $l\in [\![f]\!]$. Let $l\in [\![f]\!]$. There are some $\eta\in [\![f]\!]^{\times}$ and $(e_u)_{u\in U}\in\mathbb{N}_0^U$ such that $l=\eta\prod_{u\in U} u^{e_u}$. Clearly, $\min\{\mathrm{v}_p(l(x))\mid x\in R\}=\mathrm{v}_p(l(w))=\sum_{u\in U} e_u\mathrm{v}_p(u(w))$ for some $w\in T$. By Dickson's theorem (see \cite[Theorem 1.5.3]{GHK}) we can find some $y\in S$ such that $\mathrm{v}_p(u(y))\leq \mathrm{v}_p(u(w))$ for all $u\in U$. Since $\mathrm{v}_p(l(y))=\sum_{u\in U} e_u\mathrm{v}_p(u(y))\leq\sum_{u\in U} e_u\mathrm{v}_p(u(w))=\mathrm{v}_p(l(w))$ we infer that $\min\{\mathrm{v}_p(l(x))\mid x\in R\}=\mathrm{v}_p(l(y))=\min\{\mathrm{v}_p(l(x))\mid x\in S\}$.

Set $z=\max\{\mathrm{v}_p(g(x))\mid x\in S\}$. Then $z\in\mathbb{N}_0$. Let $k\in [\![f]\!]$. Now we prove that $\mathrm{v}_p(\frac{d(gk)}{d(k)})\leq z$. There is some $v\in S$ such that $\min\{\mathrm{v}_p(k(x))\mid x\in S\}=\mathrm{v}_p(k(v))$. We have $\mathrm{v}_p(\frac{d(gk)}{d(k)})=\min\{\mathrm{v}_p((gk)(x))\mid x\in R\}-\min\{\mathrm{v}_p(k(x))\mid x\in R\}=\min\{\mathrm{v}_p((gk)(x))\mid x\in S\}-\min\{\mathrm{v}_p(k(x))\mid x\in S\}\leq \mathrm{v}_p((gk)(v))-\mathrm{v}_p(k(v))=\mathrm{v}_p(g(v))\leq z$.

Set $P=\{p\in Q\mid\mathrm{v}_p(d(f))>0\}$. Then $P$ is finite. For every $h\in [\![f]\!]$ it follows that $\{p\in Q\mid\mathrm{v}_p(\frac{d(gh)}{d(h)})>0\}\subseteq P$. This implies that $0\not=\prod_{p\in P} p^{\max\{\mathrm{v}_p(\frac{d(gh)}{d(h)})\mid h\in [\![f]\!]\}}\in {\rm LCM}_R(\{\frac{d(gh)}{d(h)}\mid h\in [\![f]\!]\})$.

Note that $d(f)\in [\![f]\!]$. Consequently, $P\subseteq [\![f]\!]$, and thus $\prod_{p\in P} p^{\max\{\mathrm{v}_p(\frac{d(gh)}{d(h)})\mid h\in [\![f]\!]\}}\in [\![f]\!]$. Since least common multiples are unique up to units it follows immediately that ${\rm LCM}_R(\{\frac{d(gh)}{d(h)}\mid h\in [\![f]\!]\})\subseteq [\![f]\!]$.
\end{proof}

Let $R$ be a factorial domain, $Q$ a system of representatives of $\mathcal{A}(R)$ and $f\in {\rm Int}(R)^{\bullet}$. Then let $e_{f,Q}:[\![f]\!]\rightarrow R^{\bullet}$ be defined by $e_{f,Q}(g)=\prod_{p\in Q} p^{\max\{\mathrm{v}_p(\frac{d(gh)}{d(h)})\mid h\in [\![f]\!]\}}$ for all $g\in [\![f]\!]$. Observe that ${\rm LCM}_R(\{\frac{d(gh)}{d(h)}\mid h\in [\![f]\!]\})=e_{f,Q}(g)R^{\times}$ for all $g\in [\![f]\!]$. It follows from Lemma \ref{LCM} that $e_{f,Q}$ is well defined and $e_{f,Q}([\![f]\!])\subseteq [\![f]\!]$. In the following we suppose that a fixed $Q$ is given and set $e_f=e_{f,Q}$.

\bigskip
A well-known and very basic result in ring theory is that contractions of ideals to subrings are ideals again. In analogy, it holds that contractions of $s$-ideals of monoids to submonoids are $s$-ideals again. The system of $v$-ideals has a very different behavior. In the last section of this work we show that the contraction of a $v$-ideal of a Krull monoid to a monadic submonoid can fail to be a $v$-ideal. The next result, however, gives a positive answer under more restrictive conditions.

\begin{Pro}\label{contraction} Let $R$ be a factorial domain, $K$ a field of quotients of $R$, $X$ an indeterminate over $K$, $f\in R[X]^{\bullet}$ and $g\in [\![f]\!]$. Then $gK[X]\cap [\![f]\!]=\frac{g}{e_f(g)}[\![f]\!]\cap [\![f]\!]$. In particular, $gK[X]\cap [\![f]\!]\in\mathcal{I}_v([\![f]\!])$, and if $g\in\mathcal{A}(K[X])$, then $gK[X]\cap [\![f]\!]\in\mathfrak{X}([\![f]\!])$.
\end{Pro}

\begin{proof} Set $e=e_f(g)$. ``$\subseteq$'': Let $z\in gK[X]\cap [\![f]\!]$. Then there are some $a,b\in R^{\bullet}$ and some $h^{\prime}\in {\rm Int}(R)$ such that $d(h^{\prime})=1$, ${\rm GCD}_R(a,b)=R^{\times}$ and $z=\frac{agh^{\prime}}{b}$. Clearly, $h^{\prime}h^{\prime\prime}=bf^k$ for some $k\in\mathbb{N}$ and some $h^{\prime\prime}\in {\rm Int}(R)$. There are some primitive $y^{\prime}\in R[X]$, some $y^{\prime\prime}\in R[X]$ and some $d\in R$ such that $h^{\prime}=\frac{y^{\prime}}{d(y^{\prime})}$ and $h^{\prime\prime}=\frac{y^{\prime\prime}}{d}$. We infer that $y^{\prime}y^{\prime\prime}=d(y^{\prime})dbf^k$. Since $y^{\prime}$ is primitive it follows that $y^{\prime}\mid_{R[X]} f^k$, and thus $y^{\prime}\in [\![f]\!]$. Therefore, $h^{\prime}=\frac{y^{\prime}}{d(y^{\prime})}\in [\![f]\!]$. Since $bz=agh^{\prime}$, we have $bd(z)\simeq_R ad(gh^{\prime})$. This implies that $a\mid_R d(z)$ and $b\mid_R d(gh^{\prime})=\frac{d(gh^{\prime})}{d(h^{\prime})}\mid_R e$. By Lemma \ref{LCM} we have $e\in [\![f]\!]$. It follows that $a,\frac{e}{b}\in [\![f]\!]$, hence $\frac{e}{g}z=\frac{e}{b}ah^{\prime}\in [\![f]\!]$. Consequently, $z\in\frac{g}{e}[\![f]\!]\cap [\![f]\!]$. ``$\supseteq$'': Trivial.

Observe that $\frac{g}{e}$ is an element of the quotient group of $[\![f]\!]$. Therefore, $gK[X]\cap [\![f]\!]$ is an intersection of fractional principal ideals of $[\![f]\!]$, and thus $gK[X]\cap [\![f]\!]\in\mathcal{I}_v([\![f]\!])$.

Now let $g\in\mathcal{A}(K[X])$. Clearly, $gK[X]\in {\rm spec}(K[X])$, and thus $gK[X]^{\bullet}\in s$-${\rm spec}(K[X]^{\bullet})\setminus\{\emptyset\}$. This implies that $gK[X]\cap [\![f]\!]=gK[X]^{\bullet}\cap [\![f]\!]\in s$-${\rm spec}([\![f]\!])\setminus\{\emptyset\}$. Since $[\![f]\!]$ is a Krull monoid and $gK[X]\cap [\![f]\!]$ is divisorial, we have $gK[X]\cap [\![f]\!]\in\mathfrak{X}([\![f]\!])$.
\end{proof}

As a consequence, we obtain a description of the set of height-one prime ideals (of monadic submonoids generated by some $f\in R[X]^{\bullet}$) that do not contain constant elements.

\begin{Co}\label{charemptsect} Let $R$ be a factorial domain, $K$ a field of quotients of $R$, $X$ an indeterminate over $K$ and $f\in R[X]^{\bullet}$. Then $\{P\in\mathfrak{X}([\![f]\!])\mid P\cap R=\emptyset\}=\{gK[X]\cap [\![f]\!]\mid g\in [\![f]\!]\cap\mathcal{A}(R[X])\setminus R\}=\{gK[X]\cap [\![f]\!]\mid g\in [\![f]\!]\cap\mathcal{A}(K[X])\}$. In particular, if $\mathcal{R}$ is a system of representatives of $[\![f]\!]\cap\mathcal{A}(R[X])\setminus R$, then $Q:\mathcal{R}\rightarrow\{P\in\mathfrak{X}([\![f]\!])\mid P\cap R=\emptyset\}$ defined by $Q(t)=tK[X]\cap [\![f]\!]$ is a bijection.
\end{Co}

\begin{proof} Since $R$ is factorial it follows that $\mathcal{A}(R[X])\setminus R\subseteq\mathcal{A}(K[X])$.

First we prove that $\{P\in\mathfrak{X}([\![f]\!])\mid P\cap R=\emptyset\}\subseteq\{gK[X]\cap [\![f]\!]\mid g\in [\![f]\!]\cap\mathcal{A}(R[X])\setminus R\}$. Let $P\in\mathfrak{X}([\![f]\!])$ be such that $P\cap R=\emptyset$. There are some $a\in [\![f]\!]\cap R$, some $n\in\mathbb{N}$ and some finite sequence $\underline{f}\in ([\![f]\!]\cap\mathcal{A}(R[X])\setminus R)^n$ such that $f=a\prod_{i=1}^n f_i$. Clearly, $f\in P$. Consequently, there is some $i\in [1,n]$ such that $f_i\in P$. It follows by Proposition \ref{contraction} that $f_iK[X]\cap [\![f]\!]\in\mathfrak{X}([\![f]\!])$. By Proposition \ref{contraction} we have $f_iK[X]\cap [\![f]\!]=\frac{f_i}{e_f(f_i)}[\![f]\!]\cap [\![f]\!]$. Let $z\in f_iK[X]\cap [\![f]\!]$. Then $e_f(f_i)z\in f_i[\![f]\!]\subseteq P$, and thus $z\in P$. This implies that $f_iK[X]\cap [\![f]\!]\subseteq P$, hence $P=f_iK[X]\cap [\![f]\!]$.

It is obvious that $\{gK[X]\cap [\![f]\!]\mid g\in [\![f]\!]\cap\mathcal{A}(R[X])\setminus R\}\subseteq\{gK[X]\cap [\![f]\!]\mid g\in [\![f]\!]\cap\mathcal{A}(K[X])\}$.

Finally we show that $\{gK[X]\cap [\![f]\!]\mid g\in [\![f]\!]\cap\mathcal{A}(K[X])\}\subseteq\{P\in\mathfrak{X}([\![f]\!])\mid P\cap R=\emptyset\}$. Let $g\in [\![f]\!]\cap\mathcal{A}(K[X])$. Set $P=gK[X]\cap [\![f]\!]$. It follows by Proposition \ref{contraction} that $P\in\mathfrak{X}([\![f]\!])$. Since $g\not\in K$ we have $gK[X]\cap R=\{0\}$, hence $P\cap R=\emptyset$.
\end{proof}

One method of determining the divisor-class group of a Krull monoid is to identify the $v$-product decompositions of certain principal ideals into height-one prime ideals. Next, we present a useful tool which can be used for that purpose.

\begin{Pro}\label{products} Let $R$ be a factorial domain, $K$ a field of quotients of $R$, $X$ an indeterminate over $K$, $f\in R[X]^{\bullet}$ and $L$ a quotient group of $[\![f]\!]$. For $g\in [\![f]\!]$ set $P_g=gK[X]\cap [\![f]\!]$.
\begin{itemize}
\item[\textnormal{\textbf{1.}}] $(P_gP_h)_v=P_{gh}$ for all $g,h\in [\![f]\!]$.
\item[\textnormal{\textbf{2.}}] If $g\in [\![f]\!]\cap\mathcal{A}(K[X])$, then $\mathrm{v}_{P_g}(x[\![f]\!])=\mathrm{v}_g(x)$ for all $x\in L$.
\item[\textnormal{\textbf{3.}}] If $g\in [\![f]\!]$, $Q\in\mathfrak{X}([\![f]\!])$ and $q\in Q\cap\mathcal{A}(R)$, then $\mathrm{v}_Q(g[\![f]\!])\leq\mathrm{v}_q(e_f(g))$.
\end{itemize}
\end{Pro}

\begin{proof} Let $\mathcal{R}$ be a system of representatives of $[\![f]\!]\cap\mathcal{A}(R[X])\setminus R$. Note that $\mathcal{R}$ is finite.

\textbf{1.} First we show by induction that for all $k\in\mathbb{N}_0$ and $t\in\mathcal{R}$ it follows that $(P_t^k)_v=P_{t^k}$. The assertion is clear for $k=0$. Now let $k\in\mathbb{N}_0$ and $t\in\mathcal{R}$. Since $P_{t^{k+1}}$ is a divisorial ideal of $[\![f]\!]$ (by Proposition \ref{contraction}) it follows that $(P_t^{k+1})_v\subseteq P_{t^{k+1}}\subsetneq P_{t^k}=(P_t^k)_v$. We infer that $(P_t^{k+1})_v=P_{t^{k+1}}$.

Next we prove that $P_{\prod_{t\in\mathcal{R}} t^{n_t}}=(\prod_{t\in\mathcal{R}} P_t^{n_t})_v$ for all $(n_t)_{t\in\mathcal{R}}\in\mathbb{N}_0^{\mathcal{R}}$. Let $(n_t)_{t\in\mathcal{R}}\in\mathbb{N}_0^{\mathcal{R}}$. Observe that $P_s$ and $P_t$ are $v$-coprime for all distinct $s,t\in\mathcal{R}$. Since $[\![f]\!]$ is a Krull monoid we have $(\prod_{t\in\mathcal{R}} P_t^{n_t})_v=\bigcap_{t\in\mathcal{R}} (P_t^{n_t})_v=\bigcap_{t\in\mathcal{R}} P_{t^{n_t}}=(\bigcap_{t\in\mathcal{R}} t^{n_t}K[X])\cap [\![f]\!]=P_{\prod_{t\in\mathcal{R}} t^{n_t}}$.

Now let $g,h\in [\![f]\!]$. It follows from Lemma \ref{represent} that there are some $(n_t)_{t\in\mathcal{R}},(m_t)_{t\in\mathcal{R}}\in\mathbb{N}_0^{\mathcal{R}}$ such that $g\simeq_{K[X]}\prod_{t\in\mathcal{R}} t^{n_t}$ and $h\simeq_{K[X]}\prod_{t\in\mathcal{R}} t^{m_t}$. Therefore, $P_{gh}=P_{\prod_{t\in\mathcal{R}} t^{n_t+m_t}}=(\prod_{t\in\mathcal{R}} P_t^{n_t+m_t})_v=((\prod_{t\in\mathcal{R}} P_t^{n_t})_v(\prod_{t\in\mathcal{R}} P_t^{m_t})_v)_v=(P_{\prod_{t\in\mathcal{R}} t^{n_t}}P_{\prod_{t\in\mathcal{R}} t^{m_t}})_v=(P_gP_h)_v$.

\textbf{2.} Let $g\in [\![f]\!]\cap\mathcal{A}(K[X])$ and $x\in L$. There are some $y,z\in [\![f]\!]$ such that $x=\frac{y}{z}$. By 1 we have $y\in g^{\mathrm{v}_g(y)}K[X]\cap [\![f]\!]=(P_g^{\mathrm{v}_g(y)})_v$ and $y\not\in g^{\mathrm{v}_g(y)+1}K[X]\cap [\![f]\!]=(P_g^{\mathrm{v}_g(y)+1})_v$, and thus $\mathrm{v}_{P_g}(y[\![f]\!])=\mathrm{v}_g(y)$. Analogously, it follows that $\mathrm{v}_{P_g}(z[\![f]\!])=\mathrm{v}_g(z)$, hence $\mathrm{v}_{P_g}(x[\![f]\!])=\mathrm{v}_{P_g}(y[\![f]\!])-\mathrm{v}_{P_g}(z[\![f]\!])=\mathrm{v}_g(y)-\mathrm{v}_g(z)=\mathrm{v}_g(x)$.

\textbf{3.} Let $g\in [\![f]\!]$, $Q\in\mathfrak{X}([\![f]\!])$ and $q\in Q\cap\mathcal{A}(R)$. It is an easy consequence of Proposition \ref{contraction}, Corollary \ref{charemptsect}, 1 and 2 that $(\prod_{P\in\mathfrak{X}([\![f]\!]),P\cap R=\emptyset} P^{\mathrm{v}_P(g[\![f]\!])})_v=(\prod_{t\in\mathcal{R}} P_t^{\mathrm{v}_{P_t}(g[\![f]\!])})_v=(\prod_{t\in\mathcal{R}} P_t^{\mathrm{v}_t(g)})_v=P_{\prod_{t\in\mathcal{R}} t^{\mathrm{v}_t(g)}}=P_g=\frac{g}{e_f(g)}[\![f]\!]\cap [\![f]\!]=(\frac{e_f(g)}{g}[\![f]\!]\cup [\![f]\!])^{-1}$. Therefore, $(\prod_{P\in\mathfrak{X}([\![f]\!]),P\cap R\not=\emptyset} P^{\mathrm{v}_P(g[\![f]\!])})_v=g(\frac{e_f(g)}{g}[\![f]\!]\cup [\![f]\!])_v=(g[\![f]\!]\cup e_f(g)[\![f]\!])_v$. We infer that $\mathrm{v}_Q(g[\![f]\!])=\min\{\mathrm{v}_Q(g[\![f]\!]),\mathrm{v}_Q(e_f(g)[\![f]\!])\}$. Let $\mathcal{S}$ be a system of representatives of $[\![f]\!]\cap\mathcal{A}(R)$ such that $q\in\mathcal{S}$. It follows from Lemma \ref{general}.3 and \ref{general}.4 that $\mathrm{v}_Q(g[\![f]\!])\leq \mathrm{v}_Q(e_f(g)[\![f]\!])=\mathrm{v}_Q(\prod_{p\in\mathcal{S}} (p[\![f]\!])^{\mathrm{v}_p(e_f(g))})=\sum_{p\in\mathcal{S}}\mathrm{v}_p(e_f(g))\mathrm{v}_Q(p[\![f]\!])=\sum_{p\in\mathcal{S}}\mathrm{v}_p(e_f(g))\delta_{p,q}=\mathrm{v}_q(e_f(g))$.
\end{proof}

It is clear that every principal ideal of a monoid is a $v$-ideal. The converse is, of course, far from true. In the last part of this section we describe when the $v$-ideals in Proposition \ref{contraction} are principal.

\begin{Pro}\label{charprin} Let $R$ be a factorial domain, $K$ a field of quotients of $R$, $X$ an indeterminate over $K$, $f\in R[X]^{\bullet}$, and $g\in [\![f]\!]$. The following are equivalent:
\begin{itemize}
\item[\textnormal{\textbf{1.}}] $gK[X]\cap [\![f]\!]$ is a principal ideal of $[\![f]\!]$.
\item[\textnormal{\textbf{2.}}] $gK[X]\cap [\![f]\!]=\frac{g}{d(g)}[\![f]\!]$.
\item[\textnormal{\textbf{3.}}] $d(gh)=d(g)d(h)$ for all $h\in [\![f]\!]$.
\end{itemize}
\end{Pro}

\begin{proof} Set $e=e_f(g)$. By Proposition \ref{contraction} we have $gK[X]\cap [\![f]\!]=\frac{g}{e}[\![f]\!]\cap [\![f]\!]$.

$\textbf{1.}\Rightarrow\textbf{2.}$: There is some $a\in [\![f]\!]$ such that $gK[X]\cap [\![f]\!]=a[\![f]\!]$. There are some $h^{\prime},h^{\prime\prime}\in [\![f]\!]$ such that $a=\frac{gh^{\prime}}{e}$ and $\frac{g}{d(g)}=ah^{\prime\prime}$. Therefore, $e=d(g)h^{\prime}h^{\prime\prime}$, and thus $a=\frac{g}{d(g)h^{\prime\prime}}$. This implies that $d(g)h^{\prime\prime}\mid_R d(g)$, hence $h^{\prime\prime}\in [\![f]\!]^{\times}$. Consequently, $gK[X]\cap [\![f]\!]=\frac{g}{d(g)}[\![f]\!]$.

$\textbf{2.}\Rightarrow\textbf{3.}$: Let $h\in [\![f]\!]$. Since $d(g)d(h)\mid_R d(gh)$, it is sufficient to show that $\mathrm{v}_p(d(gh))\leq\mathrm{v}_p(d(g))+\mathrm{v}_p(d(h))$ for all $p\in\mathcal{A}(R)$. Let $p\in\mathcal{A}(R)$. There is some $a\in [\![f]\!]$ such that $\mathrm{v}_p(\frac{d(ga)}{d(a)})=\mathrm{v}_p(e)$. Observe that $\frac{ga}{d(a)p^{\mathrm{v}_p(e)}}\in gK[X]\cap [\![f]\!]=\frac{g}{d(g)}[\![f]\!]$. Therefore, there is some $k\in [\![f]\!]$ such that $ad(g)=kp^{\mathrm{v}_p(e)}d(a)$. This implies that $d(a)d(g)\simeq_R d(k)p^{\mathrm{v}_p(e)}d(a)$. Consequently, $\mathrm{v}_p(d(g))\geq\mathrm{v}_p(e)\geq\mathrm{v}_p(\frac{d(gh)}{d(h)})$, and thus $\mathrm{v}_p(d(gh))\leq \mathrm{v}_p(d(g))+\mathrm{v}_p(d(h))$.

$\textbf{3.}\Rightarrow\textbf{1.}$: Obviously, $e\simeq_R d(g)$, hence $gK[X]\cap R=\frac{g}{e}[\![f]\!]\cap [\![f]\!]=\frac{g}{d(g)}[\![f]\!]\cap [\![f]\!]=\frac{g}{d(g)}[\![f]\!]$.
\end{proof}

\section{First main result and important preparation results}

In this section we present the first of two main results of this work. It basically states that the divisor-class group of a monadic submonoid of ${\rm Int}(R)$ generated by some $f\in R[X]^{\bullet}$ (where $R$ is a factorial domain) is torsion-free. It is pointed out in the last section of this work that the condition ``$f\in R[X]^{\bullet}$'' is crucial here, since there are monadic submonoids of ${\rm Int}(\mathbb{Z})$ generated by some $f\in {\rm Int}(\mathbb{Z})^{\bullet}$ whose divisor-class group is not torsion-free. We proceed by preparing several useful results (which might be interesting on their own) to investigate the divisor-class group of arbitrary monadic submonoids of ${\rm Int}(R)$. We were not able to give a complete description of the structure of the divisor-class group of arbitrary monadic submonoids of ${\rm Int}(R)$. However, note that they are always finitely generated abelian groups.

\begin{Th}\label{main1} Let $R$ be a factorial domain, $K$ a field of quotients of $R$, $X$ an indeterminate over $K$ and $f\in R[X]^{\bullet}$. Set $r=|\{P\in\mathfrak{X}([\![f]\!])\mid P\cap R\not=\emptyset\}|-|\{u[\![f]\!]\mid u\in\mathcal{A}([\![f]\!])\cap R\}|$. Then $\mathcal{C}_v([\![f]\!])\cong\mathbb{Z}^r$, $|\{[P]\mid P\in\mathfrak{X}([\![f]\!])\}|\geq |\{P\in\mathfrak{X}([\![f]\!])\mid P\cap R\not=\emptyset,P$ is not principal$\}|$, and the following are equivalent:
\begin{itemize}
\item[\textnormal{\textbf{1.}}] $[\![f]\!]$ is factorial.
\item[\textnormal{\textbf{2.}}] $\mathcal{C}_v([\![f]\!])$ is finite.
\item[\textnormal{\textbf{3.}}] $d(gh)=d(g)d(h)$ for all $g,h\in [\![f]\!]$.
\end{itemize}
\end{Th}

\begin{proof} There is some $\mathcal{P}\subseteq\{P\in\mathfrak{X}([\![f]\!])\mid P\cap R\not=\emptyset\}$ such that $|\{Q\in\mathfrak{X}([\![f]\!])\mid u\in Q$ and $Q\not\in\mathcal{P}\}|=1$ for every constant atom $u$ of $[\![f]\!]$. Let $h:\mathbb{Z}^{\mathcal{P}}\rightarrow\mathcal{C}_v([\![f]\!])$ be defined by $h((n_P)_{P\in\mathcal{P}})=[(\prod_{P\in\mathcal{P}} P^{n_P})_v]$ for all $(n_P)_{P\in\mathcal{P}}\in\mathbb{Z}^{\mathcal{P}}$. Clearly, $h$ is a well-defined group homomorphism. We show that $h$ is a group isomorphism.

First we show that $h$ is surjective. It suffices to prove that for every $P\in\mathfrak{X}([\![f]\!])$ there is some $(n_Q)_{Q\in\mathcal{P}}\in\mathbb{Z}^{\mathcal{P}}$ such that $h((n_Q)_{Q\in\mathcal{P}})=[P]$. Let $P\in\mathfrak{X}([\![f]\!])$.

Case 1. $P\in\mathcal{P}$: Set $n_Q=\delta_{P,Q}$ for all $Q\in\mathcal{P}$. Then $h((n_Q)_{Q\in\mathcal{P}})=[P]$.

Case 2. $P\cap R\not=\emptyset$ and $P\not\in\mathcal{P}$: Clearly, there is some constant atom $u$ of $[\![f]\!]$ such that $u\in P$. Set $n_Q=-1$ if $Q\in\mathcal{P}$ and $u\in Q$ and $n_Q=0$ if $Q\in\mathcal{P}$ and $u\not\in Q$. We have $u[\![f]\!]=(\prod_{Q\in\mathfrak{X}([\![f]\!]),u\in Q} Q)_v$ by Lemma \ref{general}.3, hence $(\prod_{Q\in\mathcal{P},u\in Q} Q^{-1})_v=u^{-1}P$. Therefore, $h((n_Q)_{Q\in\mathcal{P}})=[(\prod_{Q\in\mathcal{P},u\in Q} Q^{-1})_v]=[P]$.

Case 3. $P\cap R=\emptyset$: By Corollary \ref{charemptsect} there is some $g\in [\![f]\!]\cap\mathcal{A}(K[X])$ such that $P=gK[X]\cap [\![f]\!]$. It follows from Proposition \ref{products} that $\mathrm{v}_Q(g[\![f]\!])=\delta_{P,Q}$ for all $Q\in\mathfrak{X}([\![f]\!])$ such that $Q\cap R=\emptyset$. Consequently, $g[\![f]\!]=(P\prod_{i=1}^m Q_i)_v$ for some $m\in\mathbb{N}_0$ and some sequence $\underline{Q}\in\{Q\in\mathfrak{X}([\![f]\!])\mid Q\cap R\not=\emptyset\}^m$. Since $h$ is a group homomorphism, it follows by case 1 and case 2 that $h((n_Q)_{Q\in\mathcal{P}})=[(\prod_{i=1}^m Q_i^{-1})_v]$ for some $(n_Q)_{Q\in\mathcal{P}}\in\mathbb{Z}^{\mathcal{P}}$. Since $(\prod_{i=1}^m Q_i^{-1})_v=g^{-1}P$ we infer that $h((n_Q)_{Q\in\mathcal{P}})=[P]$.

Next we show that $h$ is injective. Let $(n_P)_{P\in\mathcal{P}}\in\mathbb{Z}^{\mathcal{P}}$ be such that $(\prod_{P\in\mathcal{P}} P^{n_P})_v$ is principal. Let $\mathcal{R}$ be a system of representatives of $[\![f]\!]\cap\mathcal{A}(R[X])$. By Lemma \ref{quotient} there is some $(n_u)_{u\in\mathcal{R}}\in\mathbb{Z}^{\mathcal{R}}$ such that $(\prod_{P\in\mathcal{P}} P^{n_P})_v=\prod_{u\in\mathcal{R}} u^{n_u}[\![f]\!]$. By Corollary \ref{charemptsect} and Proposition \ref{products} we infer that $n_u=0$ for all $u\in\mathcal{R}\setminus R$.

Claim: If $M\in\mathfrak{X}([\![f]\!])$ and $w\in M\cap\mathcal{R}\cap R$, then $\mathrm{v}_M((\prod_{P\in\mathcal{P}} P^{n_P})_v)=n_w$. Let $M\in\mathfrak{X}([\![f]\!])$ and $w\in M\cap\mathcal{R}\cap R$. It follows by Lemma \ref{general}.3 and \ref{general}.4 that $\mathrm{v}_M(u[\![f]\!])=\delta_{u,w}$ for all $u\in\mathcal{R}\cap R$. Therefore, $\mathrm{v}_M((\prod_{P\in\mathcal{P}} P^{n_P})_v)=\mathrm{v}_M(\prod_{u\in\mathcal{R}} u^{n_u}[\![f]\!])=\sum_{u\in\mathcal{R}} n_u\mathrm{v}_M(u[\![f]\!])=\sum_{u\in\mathcal{R}\cap R} n_u\mathrm{v}_M(u[\![f]\!])=\sum_{u\in\mathcal{R}\cap R} n_u\delta_{u,w}=n_w$.

We have to show that $n_M=0$ for all $M\in\mathcal{P}$. Let $M\in\mathcal{P}$. Obviously, there is some $w\in M\cap\mathcal{R}\cap R$. It is clear that there is some $Q\in\mathfrak{X}([\![f]\!])$ such that $w\in Q$ and $Q\not\in\mathcal{P}$. It follows by the claim that $n_M=\mathrm{v}_M((\prod_{P\in\mathcal{P}} P^{n_P})_v)=n_w=\mathrm{v}_Q((\prod_{P\in\mathcal{P}} P^{n_P})_v)=0$.

Observe that $|\mathcal{P}|=r$. This implies that $\mathcal{C}_v([\![f]\!])\cong\mathbb{Z}^{\mathcal{P}}\cong\mathbb{Z}^{|\mathcal{P}|}=\mathbb{Z}^r$.

Set $\mathcal{S}=\{P\in\mathfrak{X}([\![f]\!])\mid P\cap R\not=\emptyset,P$ is not principal$\}$. To show that $|\{[P]\mid P\in\mathfrak{X}([\![f]\!])\}|\geq |\mathcal{S}|$ it is sufficient to show that for all $P,Q\in\mathcal{S}$ such that $h^{-1}([P])=h^{-1}([Q])$ it follows that $P=Q$. Let $P,Q\in\mathcal{S}$ be such that $h^{-1}([P])=h^{-1}([Q])$. Note that if $P\in\mathcal{P}$, then $h^{-1}([P])=(\delta_{P,M})_{M\in\mathcal{P}}$. Moreover, if $P\not\in\mathcal{P}$ and $u\in P\cap\mathcal{A}(R)$, then $h^{-1}([P])_M=-1$ if $M\in\mathcal{P}$ and $u\in M$ and $h^{-1}([P])_M=0$ if $M\in\mathcal{P}$ and $u\not\in M$. In particular, if $P\not\in\mathcal{P}$, then $(h^{-1}([P]))_M\not=1$ for all $M\in\mathcal{P}$. Therefore, we have either that $P,Q\in\mathcal{P}$ or $P,Q\not\in\mathcal{P}$.

Case 1. $P,Q\in\mathcal{P}$: Observe that $(\delta_{P,M})_{M\in\mathcal{P}}=h^{-1}([P])=h^{-1}([Q])=(\delta_{Q,M})_{M\in\mathcal{P}}$, hence $P=Q$.

Case 2. $P,Q\not\in\mathcal{P}$: There are some $u\in P\cap\mathcal{A}(R)$ and $v\in Q\cap\mathcal{A}(R)$. Since $h^{-1}([P])=h^{-1}([Q])$, we infer that for all $M\in\mathcal{P}$, $u\in M$ if and only if $v\in M$. Since $P$ and $Q$ are not principal and $u[\![f]\!]$ and $v[\![f]\!]$ are radical ideals of $[\![f]\!]$ (by Lemma \ref{general}.3), there are some $P^{\prime}\in\mathfrak{X}([\![f]\!])\setminus\{P\}$ and $Q^{\prime}\in\mathfrak{X}([\![f]\!])\setminus\{Q\}$ such that $u\in P^{\prime}$ and $v\in Q^{\prime}$. It is immediately clear that $P^{\prime},Q^{\prime}\in\mathcal{P}$. It follows that $u,v\in P^{\prime}$, and thus $u\simeq_{[\![f]\!]} v$ by Lemma \ref{general}.4. The choice of $\mathcal{P}$ immediately implies that $P=Q$. Finally, we prove the equivalence.

$\textbf{1.}\Rightarrow\textbf{2.}$: Trivial.

$\textbf{2.}\Rightarrow\textbf{3.}$: This is an immediate consequence of Proposition \ref{chards}.

$\textbf{3.}\Rightarrow\textbf{1.}$: Since $[\![f]\!]$ is a Krull monoid it is sufficient to show that every $P\in\mathfrak{X}([\![f]\!])$ is principal. Let $P\in\mathfrak{X}([\![f]\!])$.

Case 1. $P\cap R\not=\emptyset$: It follows by Proposition \ref{chards} that $P$ is principal.

Case 2. $P\cap R=\emptyset$: By Corollary \ref{charemptsect} there is some $g\in [\![f]\!]\cap\mathcal{A}(K[X])$ such that $P=gK[X]\cap [\![f]\!]$. Therefore, $P$ is principal by Proposition \ref{charprin}.
\end{proof}

We continue with a few result concerning the structure of height-one prime ideals.

\begin{Le}\label{divisorclosed} Let $R$ be a factorial domain and $f\in {\rm Int}(R)^{\bullet}$. Then $\{[\![a]\!]\mid a\in [\![f]\!]\}=\{[\![f]\!]\setminus\bigcup_{P\in\mathcal{P}} P\mid\mathcal{P}\subseteq\mathfrak{X}([\![f]\!])\}$ is the set of divisor-closed submonoids of $[\![f]\!]$.
\end{Le}

\begin{proof} Note that $[\![f]\!]$ is a Krull monoid. Therefore, $\{\bigcup_{P\in\mathcal{P}} P\mid\mathcal{P}\subseteq\mathfrak{X}([\![f]\!])\}$ is the set of prime $s$-ideals of $[\![f]\!]$. It is well known that the set of prime $s$-ideals of $[\![f]\!]$ is finite. Consequently, the assertion follows from \cite[Lemma 2.2.1 and 2.2.3]{GHK}.
\end{proof}

\begin{Le}\label{heightonesub} Let $T$ be a Krull monoid and $H\subseteq T$ a saturated submonoid. Then for every $P\in\mathfrak{X}(H)$ there is some $Q\in\mathfrak{X}(T)$ such that $P=Q\cap H$. In particular, $\mathfrak{X}(H)$ is the set of minimal elements of $\{Q\cap H\mid Q\in\mathfrak{X}(T)\}\setminus\{\emptyset\}$.
\end{Le}

\begin{proof} Let $P\in\mathfrak{X}(H)$. Clearly, $H$ is a Krull monoid, and thus $P$ is a divisorial ideal of $H$. Therefore, \cite[Proposition 2.4.2.3]{GHK} implies that $P_{v_T}\cap H=P$. There is some $n\in\mathbb{N}$, some injective sequence $\underline{Q}\in\mathfrak{X}(T)^n$ and some $\underline{n}\in\mathbb{N}^n$ such that $P_{v_T}=\prod_{i=1}^n (Q_i^{n_i})_{v_T}=\bigcap_{i=1}^n (Q_i^{n_i})_{v_T}$. This implies that $P=\bigcap_{i=1}^n ((Q_i^{n_i})_{v_T}\cap H)$, and thus $P=(Q_j^{n_j})_{v_T}\cap H$ for some $j\in [1,n]$. Set $Q=Q_j$. We infer that $P=\sqrt[H]{(Q^{n_j})_{v_T}\cap H}=\sqrt[T]{(Q^{n_j})_{v_T}}\cap H=Q\cap H$.
\end{proof}

\begin{Le}\label{interprep} Let $H$ be a monoid, $\mathfrak{I}$ a finite non-empty set and $(L_I)_{I\in\mathfrak{I}}$ a family of non-empty $s$-ideals of $H$ such that $(L_I\cup L_J)_v=H$ for all distinct $I,J\in\mathfrak{I}$. Then $(\bigcap_{I\in\mathfrak{I}} L_I)_v=\bigcap_{I\in\mathfrak{I}} (L_I)_v$.
\end{Le}

\begin{proof} Observe that $\bigcap_{I\in\mathfrak{I}} (L_I)_v=(\prod_{I\in\mathfrak{I}} (L_I)_v)_v=(\prod_{I\in\mathfrak{I}} L_I)_v\subseteq (\bigcap_{I\in\mathfrak{I}} L_I)_v\subseteq\bigcap_{I\in\mathfrak{I}} (L_I)_v$.
\end{proof}

The next result has some well-known analogue in the case of extension of Dedekind domains (e.g. see \cite[Proposition 2.10.2]{RR}).

\begin{Pro}\label{exponent} Let $T$ be a Krull monoid, $H\subseteq T$ a saturated submonoid, $I\in\mathcal{I}_v(H)^{\bullet}$, and $P\in\mathfrak{X}(H)$. Then $\mathrm{v}_P(I)=\max\{\lceil\frac{\mathrm{v}_Q(I_{v_T})}{\mathrm{v}_Q(P_{v_T})}\rceil\mid Q\in\mathfrak{X}(T),Q\cap H=P\}$.
\end{Pro}

\begin{proof} Without restriction let $I\subseteq P$. Set $m=\max\{\lceil\frac{\mathrm{v}_Q(I_{v_T})}{\mathrm{v}_Q(P_{v_T})}\rceil\mid Q\in\mathfrak{X}(T),Q\cap H=P\}$. First we show that $\mathrm{v}_P(I)\leq m$. Set $\mathcal{P}=\{Q\cap H\mid Q\in\mathfrak{X}(T),I\subseteq Q\}$. By Lemma \ref{heightonesub} we have $P\in\mathcal{P}$. For $L\in\mathcal{P}$ set $I_L=\bigcap_{Q\in\mathfrak{X}(T),Q\cap H=L} (Q^{\mathrm{v}_Q(I_{v_T})})_{v_T}\cap H$. Let $M\in\mathfrak{X}(T)$ be such that $M\cap H=P$. Note that $P^m\subseteq P^{\lceil\frac{\mathrm{v}_M(I_{v_T})}{\mathrm{v}_M(P_{v_T})}\rceil}\subseteq (P_{v_T})^{\lceil\frac{\mathrm{v}_M(I_{v_T})}{\mathrm{v}_M(P_{v_T})}\rceil}\subseteq ((M^{\mathrm{v}_M(P_{v_T})})_{v_T})^{\lceil\frac{\mathrm{v}_M(I_{v_T})}{\mathrm{v}_M(P_{v_T})}\rceil}\subseteq (M^{\mathrm{v}_M(P_{v_T})\lceil\frac{\mathrm{v}_M(I_{v_T})}{\mathrm{v}_M(P_{v_T})}\rceil})_{v_T}\subseteq (M^{\mathrm{v}_M(I_{v_T})})_{v_T}$. Therefore, $(P^m)_{v_H}\subseteq (I_P)_{v_H}$, and thus $\mathrm{v}_P((I_P)_{v_H})\leq m$.

Claim: For every $L\in\mathcal{P}$ there is some $a\in\mathbb{N}_0$ such that $I_L\supseteq L^a$. Let $L\in\mathcal{P}$. Set $a=\max\{\mathrm{v}_Q(I_{v_T})\mid Q\in\mathfrak{X}(T),Q\cap H=L\}$. Then $I_L\supseteq\bigcap_{Q\in\mathfrak{X}(T),Q\cap H=L} (Q^{\mathrm{v}_Q(I_{v_T})})\cap H\supseteq\bigcap_{Q\in\mathfrak{X}(T),Q\cap H=L} L^{\mathrm{v}_Q(I_{v_T})}=L^a$.

By the claim we have $\mathrm{v}_P((I_L)_{v_H})=0$ for all $L\in\mathcal{P}\setminus\{P\}$. Moreover, it follows by the claim that $(I_L\cup I_M)_{v_H}=H$ for all distinct $L,M\in\mathcal{P}$. Observe that $I_{v_T}=\bigcap_{Q\in\mathfrak{X}(T),I\subseteq Q} (Q^{\mathrm{v}_Q(I_{v_T})})_{v_T}$, hence $I=\bigcap_{Q\in\mathfrak{X}(T),I\subseteq Q} (Q^{\mathrm{v}_Q(I_{v_T})})_{v_T}\cap H=\bigcap_{L\in\mathcal{P}} I_L$. We infer by Lemma \ref{interprep} that $I=(\bigcap_{L\in\mathcal{P}} I_L)_{v_H}=\bigcap_{L\in\mathcal{P}} (I_L)_{v_H}$, hence $\mathrm{v}_P(I)=\max\{\mathrm{v}_P((I_L)_{v_H})\mid L\in\mathcal{P}\}=\mathrm{v}_P((I_P)_{v_H})\leq m$.

Next we show that $m\leq\mathrm{v}_P(I)$. We have $I=(\prod_{L\in\mathfrak{X}(H)} L^{\mathrm{v}_L(I)})_{v_H}\supseteq\prod_{L\in\mathfrak{X}(H)} L^{\mathrm{v}_L(I)}$, hence
\[
I_{v_T}\supseteq\Big(\prod_{L\in\mathfrak{X}(H)} L^{\mathrm{v}_L(I)}\Big)_{v_T}=\Big(\prod_{L\in\mathfrak{X}(H)} (L_{v_T})^{\mathrm{v}_L(I)}\Big)_{v_T}=\Big(\prod_{L\in\mathfrak{X}(H)} \Big(\Big(\prod_{M\in\mathfrak{X}(T)} M^{\mathrm{v}_M(L_{v_T})}\Big)_{v_T}\Big)^{\mathrm{v}_L(I)}\Big)_{v_T}=
\]
\[
\Big(\prod_{L\in\mathfrak{X}(H)}\prod_{M\in\mathfrak{X}(T)} M^{\mathrm{v}_M(L_{v_T})\mathrm{v}_L(I)}\Big)_{v_T}=\Big(\prod_{M\in\mathfrak{X}(T)}\prod_{L\in\mathfrak{X}(H)} M^{\mathrm{v}_M(L_{v_T})\mathrm{v}_L(I)}\Big)_{v_T}=
\]
\[
\Big(\prod_{M\in\mathfrak{X}(T)} M^{\sum_{L\in\mathfrak{X}(H)}\mathrm{v}_M(L_{v_T})\mathrm{v}_L(I)}\Big)_{v_T}.
\]
Let $Q\in\mathfrak{X}(T)$ be such that $Q\cap H=P$. Observe that $\mathrm{v}_Q(L_{v_T})=\delta_{L,P}\mathrm{v}_Q(P_{v_T})$ for all $L\in\mathfrak{X}(H)$. We infer that $\mathrm{v}_Q(I_{v_T})\leq\sum_{L\in\mathfrak{X}(H)}\mathrm{v}_Q(L_{v_T})\mathrm{v}_L(I)=\mathrm{v}_Q(P_{v_T})\mathrm{v}_P(I)$, and thus $\mathrm{v}_P(I)\geq\lceil\frac{\mathrm{v}_Q(I_{v_T})}{\mathrm{v}_Q(P_{v_T})}\rceil$.
\end{proof}

Let $T$ be a monadic submonoid of ${\rm Int}(R)$, and let $H$ be a monadic submonoid of ${\rm Int}(R)$ that is contained in $T$. Suppose that we know how $v$-product decompositions of principal ideals (generated by atoms) in $T$ look like. Then Proposition \ref{exponent} enables us to derive all $v$-product decompositions of principal ideals (generated by atoms) in $H$.

\begin{Rem}\label{charnonemptsect} Let $R$ be a factorial domain, $f\in {\rm Int}(R)^{\bullet}$ and $\mathcal{R}$ a system of representatives of $\mathcal{A}([\![f]\!])\cap R$. Then $\{P\in\mathfrak{X}([\![f]\!])\mid P\cap R\not=\emptyset\}=\dot\bigcup_{p\in\mathcal{R}}\{P\in\mathfrak{X}([\![f]\!])\mid p\in P\}$.
\end{Rem}

\begin{proof} It is straightforward to prove the equality. That the union is a disjoint union follows from Lemma \ref{general}.4.
\end{proof}

\begin{Co}\label{exponent1} Let $R$ be a factorial domain, $f\in {\rm Int}(R)^{\bullet}$, $g\in [\![f]\!]$, $I\in\mathcal{I}_v([\![g]\!])^{\bullet}$ and $P\in\mathfrak{X}([\![g]\!])$ such that $P\cap R\not=\emptyset$. Then $\mathrm{v}_P(I)=\max\{\mathrm{v}_Q(I_{v_{[\![f]\!]}})\mid Q\in\mathfrak{X}([\![f]\!]),Q\cap [\![g]\!]=P\}$.
\end{Co}

\begin{proof} Because of Proposition \ref{exponent} it suffices to show that $\mathrm{v}_Q(P_{v_{[\![f]\!]}})=1$ for all $Q\in\mathfrak{X}([\![f]\!])$ such that $Q\cap [\![g]\!]=P$. Let $Q\in\mathfrak{X}([\![f]\!])$ be such that $Q\cap [\![g]\!]=P$. By Remark \ref{charnonemptsect} there is some $u\in\mathcal{A}([\![g]\!])\cap P\cap R$. Note that $u\in\mathcal{A}([\![f]\!])\cap Q\cap R$. It follows by Lemma \ref{general}.3 that $u[\![f]\!]$ is a radical ideal of $[\![f]\!]$. Therefore, $u\not\in (Q^2)_{v_{[\![f]\!]}}$, hence $P\nsubseteq (Q^2)_{v_{[\![f]\!]}}$, and thus $\mathrm{v}_Q(P_{v_{[\![f]\!]}})=1$.
\end{proof}

In the last part of this section we describe the structure of height-one prime ideals that contain constant elements. We make an ad hoc definition to state the following results more easily. Let $R$ be a factorial domain, $f\in {\rm Int}(R)^{\bullet}$, $a\in [\![f]\!]\cap R$, and $A\subseteq B\subseteq [\![f]\!]$. We say that $A$ is $a$-compatible if there is some $w\in [\![f]\!]$ such that $a\mid_R\frac{d(uw)}{d(w)}$ for all $u\in A$. Moreover, $A$ is called maximal $a$-compatible in $B$ if $A$ is maximal (with respect to inclusion) among the $a$-compatible subsets of $B$.

\begin{Le}\label{compatible1} Let $R$ be a factorial domain, $f\in {\rm Int}(R)^{\bullet}$, $\mathcal{S}$ a system of representatives of $\mathcal{A}([\![f]\!])\setminus R$, $p$ a constant atom of $[\![f]\!]$, and $A\subseteq [\![f]\!]$.
\begin{itemize}
\item[\textnormal{\textbf{1.}}] $A$ is $p$-compatible if and only if $(A\cup\{p\})_{v_{[\![f]\!]}}\subsetneq [\![f]\!]$.
\item[\textnormal{\textbf{2.}}] $\{P\in\mathfrak{X}([\![f]\!])\mid p\in P\}=\{(Q\cup\{p\})[\![f]\!]\mid Q\subseteq\mathcal{S},Q$ is maximal $p$-compatible in $\mathcal{S}\}$.
\end{itemize}
\end{Le}

\begin{proof} $\textbf{1.}$ Let $L$ be a quotient group of $[\![f]\!]$. First let $A$ be $p$-compatible. Then there is some $w\in [\![f]\!]$ such that $p\mid_R\frac{d(uw)}{d(w)}$ for all $u\in A$. Set $z=\frac{w}{pd(w)}$. Observe that $z\in L$. It follows by Lemma \ref{general}.1 that $zu=\frac{uw}{pd(w)}\in [\![f]\!]$ for all $u\in A$. Therefore, $z(A\cup\{p\})\subseteq [\![f]\!]$. Suppose that $z\in [\![f]\!]$. Then $\frac{1}{p}=\frac{d(w)}{pd(w)}\simeq_R d(z)\in R$, a contradiction. Consequently, $z\not\in [\![f]\!]$. We infer that $[\![f]\!]\subsetneq ([\![f]\!]:_L A\cup\{p\})$, and thus $(A\cup\{p\})_{v_{[\![f]\!]}}\subsetneq [\![f]\!]$.

Next let $(A\cup\{p\})_{v_{[\![f]\!]}}\subsetneq [\![f]\!]$. There is some $z\in L\setminus [\![f]\!]$ such that $z(A\cup\{p\})\subseteq [\![f]\!]$. Set $w=zp$. It is clear that $w\in [\![f]\!]$. Let $u\in A$. Then $\frac{wu}{p}\in [\![f]\!]$, hence $p\mid_R d(wu)$ by Lemma \ref{general}.1. Since $z\not\in [\![f]\!]$ it follows by Lemma \ref{general}.1 that $p\nmid_R d(w)$, and thus $p\mid_R\frac{d(wu)}{d(w)}$.

$\textbf{2.}$ ``$\subseteq$'': Let $P\in\mathfrak{X}([\![f]\!])$ be such that $p\in P$. Set $Q=\mathcal{S}\cap P$. It is obvious that $(Q\cup\{p\})[\![f]\!]\subseteq P$. Let $x\in P$. There is some $u\in\mathcal{A}([\![f]\!])$ such that $x\in u[\![f]\!]$.

Case 1. $u\in R$: It follows by Lemma \ref{general}.4 that $u\simeq_{[\![f]\!]} p$, and thus $x\in p[\![f]\!]\subseteq (Q\cup\{p\})[\![f]\!]$.

Case 2. $u\not\in R$: There is some $z\in\mathcal{S}$ such that $u\simeq_{[\![f]\!]} z$, hence $x\in z[\![f]\!]\subseteq (Q\cup\{p\})[\![f]\!]$.

Consequently, $P=(Q\cup\{p\})[\![f]\!]$. We have $(Q\cup\{p\})_{v_{[\![f]\!]}}\subseteq P_{v_{[\![f]\!]}}=P$, hence $Q$ is $p$-compatible by 1. Let $Q^{\prime}\subseteq\mathcal{S}$ be $p$-compatible such that $Q\subseteq Q^{\prime}$. It follows by 1 that there is some $M\in\mathfrak{X}([\![f]\!])$ such that $Q^{\prime}\cup\{p\}\subseteq M$. We infer that $P=(Q\cup\{p\})[\![f]\!]\subseteq (Q^{\prime}\cup\{p\})[\![f]\!]\subseteq M$, and thus $M=P$. Therefore, $Q^{\prime}\subseteq\mathcal{S}\cap P=Q$. This shows that $Q$ is maximal $p$-compatible in $\mathcal{S}$.

``$\supseteq$'': Let $Q\subseteq\mathcal{S}$ be maximal $p$-compatible in $\mathcal{S}$. Set $P=(Q\cup\{p\})[\![f]\!]$. It is clear that $p\in P$. It follows by 1 that $P\subseteq M$ for some $M\in\mathfrak{X}([\![f]\!])$. Assume that $P\subsetneq M$. There is some $u\in\mathcal{A}([\![f]\!])$ such that $u\in M\setminus P$. It follows by Lemma \ref{general}.4 that $u\not\in R$. Consequently, there is some $w\in\mathcal{S}$ such that $u\simeq_{[\![f]\!]} w$. Set $Q^{\prime}=Q\cup\{w\}$. It is clear that $Q^{\prime}\subseteq M$, hence $Q^{\prime}$ is $p$-compatible in $\mathcal{S}$ by 1. It follows that $w\in Q^{\prime}=Q\subseteq P$, and thus $u\in P$, a contradiction. We infer that $P=M\in\mathfrak{X}([\![f]\!])$.
\end{proof}

\begin{Rem}\label{compatible2} Let $R$ be a factorial domain, $K$ a field of quotients of $R$, $X$ an indeterminate over $K$, $f\in R[X]^{\bullet}$, $a\in [\![f]\!]\cap R$, and $A\subseteq [\![f]\!]$. Then $A$ is $a$-compatible if and only if there is some primitive $g\in [\![f]\!]\cap R[X]$ such that $a\mid_R\frac{d(ug)}{d(g)}$ for all $u\in A$.
\end{Rem}

\begin{proof} First let $A$ be $a$-compatible. Then there is some $w\in [\![f]\!]$ such that $a\mid_R\frac{d(uw)}{d(w)}$ for all $u\in A$. By Lemma \ref{represent} there are some $b,c\in [\![f]\!]\cap R$ and some primitive $g\in [\![f]\!]\cap R[X]$ such that $w=\frac{bg}{c}$. We infer that $a\mid_R\frac{d(uw)}{d(w)}=\frac{d(ug)}{d(g)}$ for all $u\in A$. The converse is trivially satisfied.
\end{proof}

\section{Construction of divisor-class groups}

In this section we present a few methods which can be used to construct ``more complicated'' divisor-class groups. We start with a few preparatory results.

\begin{Le}\label{preparation} Let $R$ be a factorial domain, $K$ a quotient field of $R$, $X$ an indeterminate over $K$ and $f,g\in R[X]^{\bullet}$ such that ${\rm GCD}_{R[X]}(f,g)=R[X]^{\times}$ and $d(rs)=d(r)d(s)$ for all $r\in [\![f]\!]$ and $s\in [\![g]\!]$. Then $[\![fg]\!]=\{xy\mid x\in [\![f]\!],y\in [\![g]\!]\}$.
\end{Le}

\begin{proof} ``$\subseteq$'': Let $z\in [\![fg]\!]$. By Lemma \ref{represent} there are some $a,b\in [\![fg]\!]\cap R$ and some primitive $h\in [\![fg]\!]\cap R[X]$ such that $h\mid_{R[X]} f^kg^k$ for some $k\in\mathbb{N}$, ${\rm GCD}_R(a,b)=R^{\times}$ and $z=\frac{bh}{a}$. Since $R[X]$ is factorial, there are some $h^{\prime},h^{\prime\prime}\in R[X]$ such that $h=h^{\prime}h^{\prime\prime}$, $h^{\prime}\mid_{R[X]} f^k$ and $h^{\prime\prime}\mid_{R[X]} g^k$. Consequently, $h^{\prime}\in [\![f]\!]$ and $h^{\prime\prime}\in [\![g]\!]$. We infer that $a\mid_R d(h)=d(h^{\prime})d(h^{\prime\prime})$, and thus there are some $a^{\prime},a^{\prime\prime}\in R$ such that $a=a^{\prime}a^{\prime\prime}$, $a^{\prime}\mid_R d(h^{\prime})$ and $a^{\prime\prime}\mid_R d(h^{\prime\prime})$. Observe that $a^{\prime}\in [\![f]\!]$ and $a^{\prime\prime}\in [\![g]\!]$. Moreover, $b\mid_R d((fg)^l)=d(f)^ld(g)^l$ for some $l\in\mathbb{N}$, and thus there are some $b^{\prime},b^{\prime\prime}\in R$ such that $b=b^{\prime}b^{\prime\prime}$, $b^{\prime}\mid_R d(f)^l$ and $b^{\prime\prime}\mid_R d(g)^l$. We infer that $b^{\prime}\in [\![f]\!]$ and $b^{\prime\prime}\in [\![g]\!]$. Set $x=\frac{b^{\prime}h^{\prime}}{a^{\prime}}$ and $y=\frac{b^{\prime\prime}h^{\prime\prime}}{a^{\prime\prime}}$. Note that $z=xy$. Since $a^{\prime}\mid_R d(h^{\prime})$ we have $x\in {\rm Int}(R)$. Obviously, $xa^{\prime}=b^{\prime}h^{\prime}\in [\![f]\!]$, and thus $x\in [\![f]\!]$. Analogously, it follows that $y\in [\![g]\!]$. ``$\supseteq$'': Trivial.
\end{proof}

\begin{Le}\label{preparation1} Let $R$ be a factorial domain, $K$ a quotient field of $R$, $X$ an indeterminate over $K$ and $f,g\in {\rm Int}(R)^{\bullet}$ such that ${\rm GCD}_{K[X]}(f,g)=K[X]^{\times}$ and ${\rm GCD}_R(f(x),g(x))=R^{\times}$ for all but finitely many $x\in R$. If $h\in [\![f]\!]$, then $[\![f]\!]\cap [\![hg]\!]=[\![h]\!]$.
\end{Le}

\begin{proof} Let $h\in [\![f]\!]$. ``$\subseteq$'': Let $z\in [\![f]\!]\cap [\![hg]\!]$. There is some $k\in\mathbb{N}$ such that $z\mid_{{\rm Int}(R)} f^k$ and $z\mid_{{\rm Int}(R)} h^kg^k$. Since ${\rm GCD}_{K[X]}(f,g)=K[X]^{\times}$, it follows that $z\mid_{K[X]} h^k$. Therefore, $zy=h^k$ for some $y\in K[X]$. Let $w\in R$ be such that ${\rm GCD}_R(f(w),g(w))=R^{\times}$. Then $z(w)\mid_R f(w)^k$ and $z(w)\mid_R h(w)^kg(w)^k$. Since ${\rm GCD}_R(z(w),g(w)^k)=R^{\times}$ we have $z(w)\mid_R h(w)^k$. We infer that $y(x)\in R$ for all but finitely many $x\in R$. Consequently, $y\in {\rm Int}(R)$, and thus $z\in [\![h]\!]$. ``$\supseteq$'': Trivial.
\end{proof}

\begin{Le}\label{preparation2} Let $R$ be a factorial domain, $K$ a quotient field of $R$, $X$ an indeterminate over $K$ and $f,g\in {\rm Int}(R)^{\bullet}$ such that ${\rm GCD}_{K[X]}(f,g)=K[X]^{\times}$ and ${\rm GCD}_R(f(z),g(z))=R^{\times}$ for some $z\in R$. Then for all $x,x^{\prime}\in [\![f]\!]$ and $y,y^{\prime}\in [\![g]\!]$ such that $xy=x^{\prime}y^{\prime}$ it follows that $x\simeq_{[\![f]\!]} x^{\prime}$ and $y\simeq_{[\![g]\!]} y^{\prime}$.
\end{Le}

\begin{proof} Let $x,x^{\prime}\in [\![f]\!]$ and $y,y^{\prime}\in [\![g]\!]$ be such that $xy=x^{\prime}y^{\prime}$. There is some $k\in\mathbb{N}$ such that $x\mid_{{\rm Int}(R)} f^k$, $x^{\prime}\mid_{{\rm Int}(R)} f^k$, $y\mid_{{\rm Int}(R)} g^k$ and $y^{\prime}\mid_{{\rm Int}(R)} g^k$. Since ${\rm GCD}_{K[X]}(f,g)=K[X]^{\times}$, we infer that ${\rm GCD}_{K[X]}(x,y^{\prime})={\rm GCD}_{K[X]}(x^{\prime},y)=K[X]^{\times}$, and thus there is some $\varepsilon\in K[X]^{\times}=K^{\times}$ such that $x=\varepsilon x^{\prime}$ and $y=\varepsilon^{-1} y^{\prime}$. Moreover, $x(z)y(z)=x^{\prime}(z)y^{\prime}(z)$ and ${\rm GCD}_R(x(z),y^{\prime}(z))={\rm GCD}_R(x^{\prime}(z),y(z))=R^{\times}$, hence there is some $r\in R^{\times}$ such that $x(z)=rx^{\prime}(z)$ and $y(z)=r^{-1}y^{\prime}(z)$. Since ${\rm GCD}_R(x^{\prime}(z),y^{\prime}(z))=R^{\times}$ we have $x^{\prime}(z)\not=0$ or $y^{\prime}(z)\not=0$.

Case 1. $x^{\prime}(z)\not=0$: Observe that $rx^{\prime}(z)=x(z)=\varepsilon x^{\prime}(z)$, and thus $\varepsilon=r$.

Case 2. $y^{\prime}(z)\not=0$: Note that $r^{-1}y^{\prime}(z)=y(z)=\varepsilon^{-1} y^{\prime}(z)$, hence $\varepsilon=r$.

In any case it follows that $\varepsilon\in R^{\times}=[\![f]\!]^{\times}=[\![g]\!]^{\times}$. This implies that $x\simeq_{[\![f]\!]} x^{\prime}$ and $y\simeq_{[\![g]\!]} y^{\prime}$.
\end{proof}

\begin{Pro}\label{twoproduct} Let $R$ be a factorial domain, $K$ a quotient field of $R$, $X$ an indeterminate over $K$ and $f,g\in R[X]^{\bullet}$ such that ${\rm GCD}_{R[X]}(f,g)=R[X]^{\times}$, ${\rm GCD}_R(f(x),g(x))=R^{\times}$ for all but finitely many $x\in R$ and $d(rs)=d(r)d(s)$ for all $r\in [\![f]\!]$ and $s\in [\![g]\!]$. If $x\in [\![f]\!]$ and $y\in [\![g]\!]$, then $\mathcal{C}_v([\![xy]\!])\cong\mathcal{C}_v([\![x]\!])\times\mathcal{C}_v([\![y]\!])$.
\end{Pro}

\begin{proof} Without restriction let $R$ be not a field. Let $x\in [\![f]\!]$ and $y\in [\![g]\!]$. We show that $[\![xy]\!]_{{\rm red}}\cong [\![x]\!]_{{\rm red}}\times [\![y]\!]_{{\rm red}}$. Let $\varphi:[\![xy]\!]_{{\rm red}}\rightarrow [\![x]\!]_{{\rm red}}\times [\![y]\!]_{{\rm red}}$ be defined by $\varphi(u[\![xy]\!]^{\times})=(v[\![x]\!]^{\times},w[\![y]\!]^{\times})$ if $u\in [\![xy]\!]$, $v\in [\![x]\!]$ and $w\in [\![y]\!]$ are such that $u=vw$.

First we show that $\varphi$ is well-defined. Let $z\in [\![xy]\!]_{{\rm red}}$. Then there is some $u\in [\![xy]\!]$ such that $z=u[\![xy]\!]^{\times}$. By Lemmas \ref{preparation} and \ref{preparation1} there are some $v\in [\![f]\!]\cap [\![xy]\!]\subseteq [\![f]\!]\cap [\![xg]\!]=[\![x]\!]$ and $w\in [\![g]\!]\cap [\![xy]\!]\subseteq [\![g]\!]\cap [\![fy]\!]=[\![y]\!]$ such that $u=vw$. Now let $u^{\prime}\in [\![xy]\!]$, $v^{\prime}\in [\![x]\!]$ and $w^{\prime}\in [\![y]\!]$ be such that $u^{\prime}=v^{\prime}w^{\prime}$ and $z=u^{\prime}[\![xy]\!]^{\times}$. There is some $\varepsilon\in [\![xy]\!]^{\times}=[\![x]\!]^{\times}=[\![y]\!]^{\times}$ such that $u=\varepsilon u^{\prime}$. Therefore, $vw=\varepsilon v^{\prime}w^{\prime}$, and thus $v\simeq_{[\![x]\!]}\varepsilon v^{\prime}\simeq_{[\![x]\!]} v^{\prime}$ and $w\simeq_{[\![y]\!]} w^{\prime}$ by Lemma \ref{preparation2}. Consequently, $(v[\![x]\!]^{\times},w[\![y]\!]^{\times})=(v^{\prime}[\![x]\!]^{\times},w^{\prime}[\![y]\!]^{\times})$.

Next we show that $\varphi$ is an injective monoid homomorphism. It is clear that $\varphi([\![xy]\!])=([\![x]\!],[\![y]\!])$. Let $z,z^{\prime}\in [\![xy]\!]_{{\rm red}}$. There are some $u,u^{\prime}\in [\![xy]\!]$, $v,v^{\prime}\in [\![x]\!]$ and $w,w^{\prime}\in [\![y]\!]$ such that $z=u[\![xy]\!]^{\times}$, $z^{\prime}=u^{\prime}[\![xy]\!]^{\times}$, $u=vw$ and $u^{\prime}=v^{\prime}w^{\prime}$. Since $uu^{\prime}=vv^{\prime}ww^{\prime}$ we infer that $\varphi(zz^{\prime})=(vv^{\prime}[\![x]\!]^{\times},ww^{\prime}[\![y]\!]^{\times})=(v[\![x]\!]^{\times},w[\![y]\!]^{\times})(v^{\prime}[\![x]\!]^{\times},w^{\prime}[\![y]\!]^{\times})=\varphi(z)\varphi(z^{\prime})$. Now let $\varphi(z)=\varphi(z^{\prime})$. Then $v[\![x]\!]=v^{\prime}[\![x]\!]$ and $w[\![y]\!]=w^{\prime}[\![y]\!]$, and thus $v\simeq_{[\![xy]\!]} v^{\prime}$ and $w\simeq_{[\![xy]\!]} w^{\prime}$. This implies that $u\simeq_{[\![xy]\!]} u^{\prime}$, hence $z=z^{\prime}$.

It is clear that $\varphi$ is surjective. We conclude by \cite[Proposition 2.1.11.2]{GHK} that $\mathcal{C}_v([\![xy]\!])\cong\mathcal{C}_v([\![xy]\!]_{{\rm red}})\cong\mathcal{C}_v([\![x]\!]_{{\rm red}})\times\mathcal{C}_v([\![y]\!]_{{\rm red}})\cong\mathcal{C}_v([\![x]\!])\times\mathcal{C}_v([\![y]\!])$.
\end{proof}

By Theorem \ref{main1} we know that if $H$ is a monadic submonoid generated by some $f\in R[X]^{\bullet}$, then $d$ is multiplicative on $H$ if and only if $H$ is factorial. The last proposition requires a less stringent form of being multiplicative. Next we show that there is an interesting class of monadic submonoids for which $d$ satisfies this weak form of being multiplicative.

\begin{Le}\label{preparation3} Let $R$ be a factorial domain, $K$ a field of quotients of $R$, $X$ an indeterminate over $K$, $a\in R$ and $f,g\in R[X]^{\bullet}$ such that ${\rm GCD}_R(f(a),g(a))=R^{\times}$ and for all $p\in\mathcal{A}(R)$ and $h\in\mathcal{A}(R[X])$ with $(p\mid_R f(a)$ and $h\mid_{R[X]} g)$ or $(p\mid_R g(a)$ and $h\mid_{R[X]} f)$ it follows that $p\mid_{R[X]} h-h(a)$. Then $d(rs)=d(r)d(s)$ for all $r\in [\![f]\!]$ and $s\in [\![g]\!]$.
\end{Le}

\begin{proof} Let $r\in [\![f]\!]$ and $s\in [\![g]\!]$. Let $\mathcal{P}$ be a system of representatives of $\mathcal{A}(R)$. To prove that $d(rs)=d(r)d(s)$, we need to show that for each $p\in\mathcal{P}$ there is some $y\in R$ such that $\mathrm{v}_p(r(y))=\min\{\mathrm{v}_p(r(x))\mid x\in R\}$ and $\mathrm{v}_p(s(y))=\min\{\mathrm{v}_p(s(x))\mid x\in R\}$. Let $p\in\mathcal{P}$. It is an easy consequence of Lemma \ref{represent} that there are some $b,c\in K^{\bullet}$, $n,m\in\mathbb{N}$, $\underline{\alpha}\in\mathbb{N}_0^n$, $\underline{\beta}\in\mathbb{N}_0^m$, $\underline{f}\in\mathcal{A}(R[X])^n$ and $\underline{g}\in\mathcal{A}(R[X])^m$ such that $r=b\prod_{i=1}^n f_i^{\alpha_i}$, $s=c\prod_{j=1}^m g_j^{\beta_j}$, $f_i\mid_{R[X]} f$ for all $i\in [1,n]$ and $g_j\mid_{R[X]} g$ for all $j\in [1,m]$.

If $z\in R$ is such that $\mathrm{v}_p(f_i(z))=0$ for all $i\in [1,n]$, then $\mathrm{v}_p(r(z))=\mathrm{v}_p(b)+\sum_{i=1}^n\alpha_i\mathrm{v}_p(f_i(z))=\mathrm{v}_p(b)\leq\mathrm{v}_p(b)+\sum_{i=1}^n\alpha_i\mathrm{v}_p(f_i(v))=\mathrm{v}_p(r(v))$ for all $v\in R$, and thus $\mathrm{v}_p(r(z))=\min\{\mathrm{v}_p(r(x))\mid x\in R\}$.

Analogously, if $w\in R$ is such that $\mathrm{v}_p(g_j(w))=0$ for all $j\in [1,m]$, then $\mathrm{v}_p(s(w))=\min\{\mathrm{v}_p(s(x))\mid x\in R\}$.

Case 1. $p\nmid_R (fg)(a)$: Observe that $\mathrm{v}_p(f_i(a))=\mathrm{v}_p(g_j(a))=0$ for all $i\in [1,n]$ and $j\in [1,m]$. Therefore, $\mathrm{v}_p(r(a))=\min\{\mathrm{v}_p(r(x))\mid x\in R\}$ and $\mathrm{v}_p(s(a))=\min\{\mathrm{v}_p(s(x))\mid x\in R\}$.

Case 2. $p\mid_R f(a)$: There is some $y\in R$ such that $\mathrm{v}_p(r(y))=\min\{\mathrm{v}_p(r(x))\mid x\in R\}$. Let $j\in [1,m]$. Note that $p\mid_{R[X]} g_j-g_j(a)$. Consequently, $p\mid_R g_j(y)-g_j(a)$, and since $p\nmid_R g_j(a)$, we have $\mathrm{v}_p(g_j(y))=0$. This implies that $\mathrm{v}_p(s(y))=\min\{\mathrm{v}_p(s(x))\mid x\in R\}$.

Case 3. $p\mid_R g(a)$: Goes along the same lines as case 2.
\end{proof}

\begin{Pro}\label{twoproduct1} Let $R$ be a factorial domain, $K$ a field of quotients of $R$, $X$ an indeterminate over $K$, $a\in R$ and $f,g\in R[X]^{\bullet}$ such that ${\rm GCD}_{R[X]}(f,g)=R[X]^{\times}$, ${\rm GCD}_R(f(a),g(a))=R^{\times}$ and $d(rs)=d(r)d(s)$ for all $r\in [\![f]\!]$ and $s\in [\![g]\!]$. Then $\mathcal{C}_v([\![fg]\!])\cong\mathcal{C}_v([\![f]\!])\times\mathcal{C}_v([\![g]\!])$.
\end{Pro}

\begin{proof} It follows from Lemmas \ref{preparation} and \ref{preparation2} that $[\![fg]\!]_{{\rm red}}\cong [\![f]\!]_{{\rm red}}\times [\![g]\!]_{{\rm red}}$. We conclude by \cite[Proposition 2.1.11.2]{GHK} that $\mathcal{C}_v([\![fg]\!])\cong\mathcal{C}_v([\![fg]\!]_{{\rm red}})\cong\mathcal{C}_v([\![f]\!]_{{\rm red}})\times\mathcal{C}_v([\![g]\!]_{{\rm red}})\cong\mathcal{C}_v([\![f]\!])\times\mathcal{C}_v([\![g]\!])$.
\end{proof}

\begin{Co}\label{twoproduct2} Let $R$ be a factorial domain, $K$ a field of quotients of $R$, $X$ an indeterminate over $K$, $a\in R$ and $f,g\in R[X]^{\bullet}$ such that ${\rm GCD}_{R[X]}(f,g)=R[X]^{\times}$, ${\rm GCD}_R(f(a),g(a))=R^{\times}$ and for all $p\in\mathcal{A}(R)$ and $h\in\mathcal{A}(R[X])$ with $(p\mid_R f(a)$ and $h\mid_{R[X]} g)$ or $(p\mid_R g(a)$ and $h\mid_{R[X]} f)$ it follows that $p\mid_{R[X]} h-h(a)$. Then $\mathcal{C}_v([\![fg]\!])\cong\mathcal{C}_v([\![f]\!])\times\mathcal{C}_v([\![g]\!])$.
\end{Co}

\begin{proof} This is an immediate consequence of Lemma \ref{preparation3} and Proposition \ref{twoproduct1}.
\end{proof}

\section{Examples, important consequences and second main result}

In this section we present several applications of the abstract theory. We start with a bunch of examples that serve as counterexamples for various questions. We use the set of prime numbers as choice for the set of representatives of $\mathcal{A}(\mathbb{Z})$. If $\mathbb{Z}$ is the base ring, then let all monadic submonoids be monadic submonoids of ${\rm Int}(\mathbb{Z})$. Note that if $H$ is an atomic monoid (e.g. $H$ is a Krull monoid), $\mathcal{Q}$ is a system of representatives of $\mathcal{A}(H)$, and $P\in\mathfrak{X}(H)$, then $P=\bigcup_{u\in\mathcal{Q}\cap P} uH$. We will use this fact without citation. It was used implicitly in the proof of Lemma \ref{compatible1}. Also note that if $T\subseteq H$ is a divisor-closed submonoid of $H$ and $z\in H$, then either $zH\cap T=zT$ or $zH\cap T=\emptyset$.

\begin{Ex}\label{Ex1} Let $X$ be an indeterminate over $\mathbb{Q}$. Set $u_1=2$, $u_2=3$, $u_3=X$, $u_4=X-1$, $u_5=X-2$, $u_6=\frac{u_3u_4}{2}$, $u_7=\frac{u_4u_5}{2}$, $u_8=\frac{u_3u_4u_5}{6}$, $u_9=\frac{u_3u_4^2u_5}{12}$ and $u_{10}=\frac{u_3u_4^3u_5}{24}$. For $J\subseteq [1,10]$, set $U_J=\{u_j\mid j\in J\}$. Set $H=[\![u_3u_4u_5]\!]$, $S=[\![\frac{u_3^2u_4^3u_5^2}{8}]\!]$, $T=[\![\frac{u_3u_4^3u_5^3}{24}]\!]$, $V=[\![\frac{u_3u_4^3u_5}{6}]\!]$, $W=[\![\frac{u_3u_4^2u_5}{12}]\!]$, $Y=[\![u_3u_4]\!]$, and $Z=[\![\frac{u_3^2u_4}{2}]\!]$.
\begin{itemize}
\item $S=H\setminus U_{\{1,4\}}H$, $T=H\setminus U_{\{1,2,3,4,6\}}H$, and $V=H\setminus U_{\{2,3,5,6,7\}}H$.
\item $W=H\setminus U_{\{1,2,3,4,5,6,7\}}H$, $Y=H\setminus U_{\{2,5,7,8,9,10\}}H$, and $Z=H\setminus U_{\{1,2,4,5,7,8,9,10\}}H$.
\item $\{uH\mid u\in\mathcal{A}(H)\}=\{u_1H,u_2H,u_3H,u_4H,u_5H,u_6H,u_7H,u_8H,u_9H,u_{10}H\}$.
\item $\{uS\mid u\in\mathcal{A}(S)\}=\{u_2S,u_3S,u_5S,u_6S,u_7S,u_8S,u_9S,u_{10}S\}$.
\item $\{uT\mid u\in\mathcal{A}(T)\}=\{u_5T,u_7T,u_8T,u_9T,u_{10}T\}$, and $\{uV\mid u\in\mathcal{A}(V)\}=\{u_1V,u_4V,u_8V,u_9V,u_{10}V\}$.
\item $\{uW\mid u\in\mathcal{A}(W)\}=\{u_8W,u_9W,u_{10}W\}$, and $\{uY\mid u\in\mathcal{A}(Y)\}=\{u_1Y,u_3Y,u_4Y,u_6Y\}$.
\item $\{uZ\mid u\in\mathcal{A}(Z)\}=\{u_3Z,u_6Z\}$.
\item $\mathfrak{X}(H)=\{U_{\{1,3,5,6,8,9\}}H,U_{\{1,3,5,7,8,9\}}H,U_{\{1,4\}}H,U_{\{2,3,6\}}H,U_{\{2,4,6,7,9,10\}}H,U_{\{2,5,7\}}H,U_{\{3,6,8,9,10\}}H,\\
U_{\{4,6,7,8,9,10\}}H,U_{\{5,7,8,9,10\}}H\}$.
\item $\mathfrak{X}(S)=\{U_{\{2,3,6\}}S,U_{\{2,5,7\}}S,U_{\{2,6,7,9,10\}}S,U_{\{3,5,6,8,9\}}S,U_{\{3,5,7,8,9\}}S,U_{\{3,6,8,9,10\}}S,U_{\{5,7,8,9,10\}}S,\\
U_{\{6,7,8,9,10\}}S\}$.
\item $\mathfrak{X}(T)=\{U_{\{5,7\}}T,U_{\{5,8,9\}}T,U_{\{7,9,10\}}T,U_{\{8,9,10\}}T\}$, and $\mathfrak{X}(V)=\{U_{\{1,4\}}V,U_{\{1,8,9\}}V,U_{\{4,9,10\}}V,\\
U_{\{8,9,10\}}V\}$.
\item $\mathfrak{X}(W)=\{U_{\{8,9\}}W,U_{\{9,10\}}W\}$, $\mathfrak{X}(Y)=\{U_{\{1,3\}}Y,U_{\{1,4\}}Y,U_{\{3,6\}}Y,U_{\{4,6\}}Y\}$, and $\mathfrak{X}(Z)=\\
\{u_3Z,u_6Z\}$.
\item $\mathcal{C}_v(H)\cong\mathcal{C}_v(S)\cong\mathbb{Z}^4$, $\mathcal{C}_v(T)\cong\mathcal{C}_v(V)\cong\mathcal{C}_v(Y)\cong\mathbb{Z}$, and $\mathcal{C}_v(W)\cong\mathbb{Z}/2\mathbb{Z}$.
\item $Z$ is factorial and $u_3\mathbb{Q}[X]\cap Z=U_{\{3,6\}}Z=Z\setminus Z^{\times}$ is not divisorial.
\end{itemize}
\end{Ex}

\begin{proof} It is straightforward to prove that $d(u_3^ku_4^lu_5^m)=2^{\min\{2k+m,l,k+2m\}}3^{\min\{k,l,m\}}$ for all $k,l,m\in\mathbb{N}_0$. Now one can show by a careful case analysis that $\{\underline{x}\in\mathbb{N}_0^3\mid\underline{x}$ is $(u_3,u_4,u_5)$-irreducible$\}=\{(1,0,0),(0,1,0),(0,0,1),(1,1,0),(0,1,1),(1,1,1),(1,2,1),(1,3,1)\}$. It follows that $\{uH\mid u\in\mathcal{A}(H)\}=\{u_iH\mid i\in [1,10]\}$, by Proposition \ref{charatoms}.

\medskip
Note that each primitive $g\in H\cap R[X]$ is associated to some element of the form $u_3^ku_4^lu_5^m$ for some $k,l,m\in\mathbb{N}_0$. We have $2\mid\frac{d(u_3u_4^2u_5)}{d(u_3u_4u_5)}$. If $k,l,m\in\mathbb{N}_0$ are such that $2\mid\frac{d(u_3^ku_4^{l+1}u_5^m)}{d(u_3^ku_4^lu_5^m)}$, then $l<\min\{k,m\}$, and then it is easy to show that $2\nmid\frac{d(u_ju_3^ku_4^lu_5^m)}{d(u_3^ku_4^lu_5^m)}$ for every $j\in\{3,5,6,7,8,9,10\}$. Consequently, $U_{\{4\}}$ is maximal $2$-compatible in $\mathcal{A}(H)\setminus\mathbb{Z}$. Analogously, we have $U_{\{1,3,5,6,8,9\}}$ and $U_{\{1,3,5,7,8,9\}}$ are maximal 2-compatible in $\mathcal{A}(H)\setminus\mathbb{Z}$, and $U_{\{2,3,6\}}$, $U_{\{2,4,6,7,9,10\}}$, and $U_{\{2,5,7\}}$ are maximal 3-compatible in $\mathcal{A}(H)\setminus\mathbb{Z}$. Clearly, we have $u_3\mathbb{Q}[X]\cap H=U_{\{3,6,8,9,10\}}H$, $u_4\mathbb{Q}[X]\cap H=U_{\{4,6,7,8,9,10\}}H$, and $u_5\mathbb{Q}[X]\cap H=U_{\{5,7,8,9,10\}}H$. It follows from Corollary \ref{charemptsect}, Remark \ref{charnonemptsect}, Lemma \ref{compatible1}.2, and Remark \ref{compatible2} that $\mathfrak{X}(H)$ can be expressed as asserted.

\medskip
It is easy to see that the other monoids ($S,T,V,W,Y$ and $Z$) are all monadic submonoids of $H$. They are, therefore, complements of unions of height-one prime ideals of $H$ by Lemma \ref{heightonesub}. We show that $S=H\setminus U_{\{1,4\}}H$. The corresponding equalities for the remaining monoids can be proved in analogy. Set $A=H\setminus U_{\{1,4\}}H$. First, note that $A=[\![h]\!]$ for some $h\in H$ by Lemma \ref{heightonesub}. Since $h$ is a product of atoms of $H$, and $h\in A$ we have $h$ is associated to a product of elements of $U_{\{2,3,5,6,7,8,9,10\}}$. Consequently, $A=[\![u_2u_3\prod_{i=5}^{10}u_i]\!]=[\![u_2(\frac{u_3u_4u_5}{2})^2(\frac{u_3u_4^2u_5}{12})^3]\!]=[\![u_2\frac{u_3u_4u_5}{2}\frac{u_3u_4^2u_5}{12}]\!]=S$. It is now simple to prove the remaining statements concerning sets of atoms. It is clear that $Z$ is factorial (since every height-one prime ideal of $Z$ is principal). Moreover, we have $u_3\mathbb{Q}[X]\cap Z=u_3\mathbb{Q}[X]\cap H\cap Z=U_{\{3,6,8,9,10\}}H\cap Z=U_{\{3,6\}}Z=Z\setminus Z^{\times}$, since every non-unit of $Z$ is a multiple of $u_3$ or $u_6$. If $Z\setminus Z^{\times}$ is divisorial, then $Z$ is a discrete valuation monoid, and hence it has only one atom up to associates, a contradiction. Therefore, $Z\setminus Z^{\times}$ is not divisorial.

\medskip
It remains to show all statements about divisor-class groups. It follows from Theorem \ref{main1} that $\mathcal{C}_v(H)\cong\mathbb{Z}^4$. We only show that $\mathcal{C}_v(W)\cong\mathbb{Z}/2\mathbb{Z}$. The other assertions follow in analogy. Let $(P_i)_{i=1}^9$ be the sequence of height-one prime ideals of $H$ in the above order (i.e., $P_1=U_{\{1,3,5,6,8,9\}}$, $P_2=U_{\{1,3,5,7,8,9\}}H$, and so on). We determine the $v$-product decompositions of principal ideals of $H$ generated by atoms. Set $f=u_3u_4u_5$. The definition of $e_f$ can be found in section 3. It is straightforward to prove that $e_f(u_4)=e_f(u_6)=e_f(u_7)=e_f(u_9)=6$. We infer by Proposition \ref{products} that $u_1H=(P_1P_2P_3)_v$, $u_2H=(P_4P_5P_6)_v$, $u_4H=(P_3P_5P_8)_v$, $u_6H=(P_1P_4P_5P_7P_8)_v$, $u_7H=(P_2P_5P_6P_8P_9)_v$, and $u_9H=(P_1P_2P_5P_7P_8^2P_9)_v$. Therefore, $u_3H=\frac{u_1u_6}{u_4}H=(P_1^2P_2P_4P_7)_v$, $u_5H=\frac{u_1u_7}{u_4}H=(P_1^2P_2^2P_6P_9)_v$, $u_8H=\frac{u_1u_9}{u_4}H=(P_1^2P_2^2P_7P_8P_9)_v$, and $u_{10}H=\frac{u_4u_9}{u_1}H=(P_5^2P_7P_8^3P_9)_v$. Set $Q_1=U_{\{8,9\}}W$, and $Q_2=U_{\{9,10\}}W$. Then $\{P\in\mathfrak{X}(H)\mid P\cap W=Q_1\}=\{P_1,P_2\}$, and $\{P\in\mathfrak{X}(H)\mid P\cap W=Q_2\}=\{P_5\}$. We have $\mathrm{v}_{P_1}((Q_1)_{v_H})=\min\{\mathrm{v}_{P_1}(u_8H),\mathrm{v}_{P_1}(u_9H)\}=1$, $\mathrm{v}_{P_2}((Q_1)_{v_H})=\min\{\mathrm{v}_{P_2}(u_8H),\mathrm{v}_{P_2}(u_9H)\}=1$, and $\mathrm{v}_{P_5}((Q_2)_{v_H})=\min\{\mathrm{v}_{P_5}(u_9H),\mathrm{v}_{P_5}(u_{10}H)\}=1$. We infer by Proposition \ref{exponent} that $\mathrm{v}_{Q_1}(u_8W)=\max\{\lceil\frac{\mathrm{v}_{P_1}(u_8H)}{\mathrm{v}_{P_1}((Q_1)_{v_H})}\rceil,\lceil\frac{\mathrm{v}_{P_2}(u_8H)}{\mathrm{v}_{P_2}((Q_1)_{v_H})}\rceil\}=2$, $\mathrm{v}_{Q_2}(u_{10}W)=2$, and $\mathrm{v}_{Q_1}(u_9W)=\mathrm{v}_{Q_2}(u_9W)=1$. Consequently, $u_8W=(Q_1^2)_{v_W}$, $u_9W=(Q_1Q_2)_{v_W}$, and $u_{10}W=(Q_2^2)_{v_W}$. This implies that $[Q_1]=[Q_2]$, and thus $[Q_1]$ is an element of order 2 which generates $\mathcal{C}_v(W)$. Therefore, $\mathcal{C}_v(W)\cong\mathbb{Z}/2\mathbb{Z}$.
\end{proof}

The last example shows that there is some factorial domain $R$ and $f,g,h,k\in {\rm Int}(R)^{\bullet}$ such that $k\in [\![h]\!]$, and
\begin{itemize}
\item $[\![f]\!]$ satisfies the equivalent conditions in Proposition \ref{chards} and yet $\mathcal{C}_v([\![f]\!])$ is infinite.
\item $\mathcal{C}_v([\![g]\!])$ is finite, and yet $[\![g]\!]$ is not factorial.
\item There are some $I\in\mathcal{I}_v([\![k]\!])^{\bullet}$ and $P\in\mathfrak{X}([\![k]\!])$ such that $\mathrm{v}_P(I)<\max\{\mathrm{v}_Q(I_{v_{[\![h]\!]}})\mid Q\in\mathfrak{X}([\![h]\!]),Q\cap [\![k]\!]=P\}$.
\end{itemize}

Observe that $U_{\{3,6,8,9,10\}}H\cap Z=U_{\{3,6\}}Z$ in the last example. We know that $H$ is a Krull monoid, $U_{\{3,6,8,9,10\}}H$ is a height-one prime ideal of $H$ (and hence it is divisorial), $Z$ is a monadic submonoid of $H$, and yet $U_{\{3,6,8,9,10\}}H\cap Z$ is not a divisorial ideal of $Z$.

Recall that if $G$ is an additive abelian group, then the Davenport constant of $G$ (denoted by ${\rm D}(G)$) is defined to be the supremum of all lengths of nonempty minimal zero-sum sequences of $G$ (see \cite[Definition 1.4.8]{GHK}).

\begin{Le}\label{exprep} Let $R$ be a factorial domain, $K$ a field of quotients of $R$, $X$ an indeterminate over $K$, $p\in\mathcal{A}(R)$, $n\in\mathbb{N}_{\geq 2}$ and $\underline{f}\in (\mathcal{A}(R[X])\setminus R)^n$ a sequence of pairwise non-associated elements of $R[X]$ such that $d(\prod_{i=1}^n f_i^{r_i})=p^{\min\{r_i\mid i\in [1,n]\}}$ for all $\underline{r}\in\mathbb{N}_0^n$. Set $H=[\![\prod_{i=1}^n f_i]\!]$. Then $\{uH\mid u\in\mathcal{A}(H)\}=\{pH,\frac{\prod_{i=1}^n f_i}{p}H\}\cup\{f_iH\mid i\in [1,n]\}$, $\mathfrak{X}(H)=\{pH\cup f_iH,f_iH\cup\frac{\prod_{j=1}^n f_j}{p}H\mid i\in [1,n]\}$, $\mathcal{C}_v(H)\cong\mathbb{Z}^{n-1}$, all proper divisor-closed submonoids of $H$ are factorial, for every $P\in\mathfrak{X}(H)$ there is some $Q\in\mathfrak{X}(H)$ such that $(PQ)_v$ is a principal ideal of $H$ and ${\rm D}(\{[P]\mid P\in\mathfrak{X}(H)\})\geq n$.
\end{Le}

\begin{proof} Set $e_i=(\delta_{i,j})_{j=1}^n$ for each $i\in [1,n]$. Observe that $\{\underline{m}\in\mathbb{N}_0^n\mid\underline{m}$ is $\underline{f}$-irreducible$\}=\{e_i\mid i\in [1,n]\}\cup\{\sum_{i=1}^n e_i\}$. Therefore, Proposition \ref{charatoms} implies that $\{uH\mid u\in\mathcal{A}(H)\}=\{pH,\frac{\prod_{i=1}^n f_i}{p}H\}\cup\{f_iH\mid i\in [1,n]\}$. Note that $\mathcal{R}=\{p\}$ is a system of representatives of $\mathcal{A}(H)\cap R$ and $\mathcal{S}=\{f_i\mid i\in [1,n]\}\cup\{\frac{\prod_{i=1}^n f_i}{p}\}$ is a system of representatives of $\mathcal{A}(H)\setminus R$. It follows by Remark \ref{charnonemptsect}, Lemma \ref{compatible1}.2, and Remark \ref{compatible2} that $\{P\in\mathfrak{X}(H)\mid P\cap R\not=\emptyset\}=\{P\in\mathfrak{X}(H)\mid p\in P\}=\{pH\cup f_iH\mid i\in [1,n]\}$. Moreover, we have $\{P\in\mathfrak{X}(H)\mid P\cap R=\emptyset\}=\{f_iK[X]\cap H\mid i\in [1,n]\}=\{f_iH\cup\frac{\prod_{j=1}^n f_j}{p}H\mid i\in [1,n]\}$ by Corollary \ref{charemptsect}. Consequently, $\mathcal{C}_v(H)\cong\mathbb{Z}^{n-1}$ by Theorem \ref{main1}.
For $i\in [1,n]$, set $S_i=H\setminus (pH\cup f_iH)$ and $T_i=H\setminus (f_iH\cup\frac{\prod_{j=1}^n f_j}{p}H)$. Let $i\in [1,n]$. By Lemma \ref{divisorclosed}, $S_i$ and $T_i$ are divisor-closed submonoids of $H$ and every proper divisor-closed submonoid of $H$ is a divisor-closed submonoid of $S_j$ or $T_j$ for some $j\in [1,n]$. Note that $\mathfrak{X}(S_i)=\{f_kS_i,\frac{\prod_{j=1}^n f_j}{p}S_i\mid k\in [1,n]\setminus\{i\}\}$ and $\mathfrak{X}(T_i)=\{pT_i,f_kT_i\mid k\in [1,n]\setminus\{i\}\}$ by Lemma \ref{heightonesub}. This implies that $S_i$ and $T_i$ are factorial (since all of their height-one prime ideals are principal), and hence every proper divisor-closed submonoid of $H$ is factorial.

It is an easy consequence of Proposition \ref{products}.3 that $f_iH=((pH\cup f_iH)(f_iH\cup\frac{\prod_{j=1}^n f_j}{p}H))_v$ for every $i\in [1,n]$. Let $P\in\mathfrak{X}(H)$. Clearly, there are some $i\in [1,n]$ and $Q\in\mathfrak{X}(H)$ such that $\{P,Q\}=\{pH\cup f_iH,f_iH\cup\frac{\prod_{j=1}^n f_j}{p}H\}$. It follows that $(PQ)_v$ is a principal ideal of $H$.

Note that $pH=(\prod_{i=1}^n pH\cup f_iH)_v$ (since $pH$ is a radical ideal of $H$). Since $p$ is an atom of $H$, it follows that no nonempty $v$-subproduct of $(\prod_{i=1}^n pH\cup f_iH)_v$ is principal. Therefore, ${\rm D}(\{[P]\mid P\in\mathfrak{X}(H)\})\geq n$.
\end{proof}

Next we recall a simple irreducibility criterion similar to Eisenstein's criterion. We include a proof for the sake of completeness. If $R$ is a factorial domain, $K$ is a field of quotients of $R$, $X$ is an indeterminate over $K$, and $f\in R[X]$ with ${\rm deg}(f)=n\in\mathbb{N}_0$, then let $(f_i)_{i=0}^n\in R^{n+1}$ be the unique sequence such that $f=\sum_{i=0}^n f_iX^i$.

\begin{Le}\label{exprep1} Let $R$ be a factorial domain, $K$ a field of quotients of $R$, $X$ an indeterminate over $K$, $p\in\mathcal{A}(R)$, and $f\in R[X]\setminus R$ primitive such that $n={\rm deg}(f)$, $p\nmid_R f_0$, $p^2\nmid_R f_n$, and $p\mid_R f_i$ for all $i\in [1,n]$. Then $f\in\mathcal{A}(R[X])$.
\end{Le}

\begin{proof} Clearly, $f\in R[X]^{\bullet}\setminus R[X]^{\times}$. Let $g,h\in R[X]$ be such that $f=gh$. Let $l={\rm deg}(g)$ and $m={\rm deg}(h)$. Without restriction we can assume that $l\leq m$. Observe that $p\mid_R g_lh_m$, $p^2\nmid_R g_lh_m$, and $m\geq 1$. We need to show that $g\in R^{\times}$.

Case 1. $p\mid_R g_l$ and $p\nmid_R h_m$: We prove that $p\mid_R g_{l-i}$ for all $i\in [0,l]$ by induction on $i$. Let $i\in [0,l]$ be such that $p\mid_R g_{l-j}$ for each $j\in [0,i-1]$. It follows that $p\mid_R f_{l+m-i}=\sum_{j=0}^i g_{l-j}h_{m+j-i}$, and hence $p\mid_R g_{l-i}h_m$. Consequently, $p\mid_R g_{l-i}$. We infer that $p\mid_{R[X]} g\mid_{R[X]} f$, and thus $p\mid_R f_0$, a contradiction.

Case 2. $p\nmid_R g_l$ and $p\mid_R h_m$: We prove that $l=0$. Assume to the contrary that $l>0$. We show by induction on $i$ that $p\mid_R h_{m-i}$ for all $i\in [0,m]$. Let $i\in [0,m]$ be such that $p\mid_R h_{m-j}$ for every $j\in [0,i-1]$. Note that $p\mid_R f_{l+m-i}=\sum_{j=0}^{\min\{i,l\}} g_{l-j}h_{m+j-i}$, and thus $p\mid_R g_lh_{m-i}$. We infer that $p\mid_R h_{m-i}$. Consequently, $p\mid_{R[X]} h\mid_{R[X]} f$, and hence $p\mid_R f_0$, a contradiction. It follows that $l=0$, and thus $g\in R$. Since $f$ is primitive we have $g\in R^{\times}$.
\end{proof}

In the beginning of this section we have provided examples of monadic submonoids of ${\rm Int}(\mathbb{Z})$ whose divisor-class group is a torsion group or torsion-free, but not trivial. Next we provide an example of a monadic submonoid of ${\rm Int}(\mathbb{Z})$ whose divisor-class group is neither torsion-free nor a torsion group.

\begin{Ex} Let $X$ be an indeterminate over $\mathbb{Q}$. Set $p_1=7$, $p_2=13$, $p_3=19$, $p_4=31$, $p_5=37$, $p_6=43$, $p_7=67$, $a=\prod_{i=1}^7 p_i$, $f=(aX+1)(aX+2)(aX+3)$, $g=\prod_{i=1}^7 (6Xf+p_i)$, and $H=[\![\frac{(aX+1)(aX+2)^2(aX+3)g}{12}]\!]$. Then $\mathcal{C}_v(H)\cong\mathbb{Z}/2\mathbb{Z}\times\mathbb{Z}^6$.
\end{Ex}

\begin{proof} It is straightforward to show that $d((aX+1)^k(aX+2)^l(aX+3)^m)=2^{\min\{2k+m,l,k+2m\}}3^{\min\{k,l,m\}}$ for all $k,l,m\in\mathbb{N}_0$. As in Example \ref{Ex1} we obtain that $\mathcal{C}_v([\![\frac{(aX+1)(aX+2)^2(aX+3)}{12}]\!])\cong\mathbb{Z}/2\mathbb{Z}$. Moreover, one can show that $d(\prod_{i=1}^7 (aXf+p_i)^{b_i})=p_1^{\min\{b_i\mid i\in [1,7]\}}$ for all $\underline{b}\in\mathbb{N}_0^7$. Therefore, it follows from Lemma \ref{exprep} that $\mathcal{C}_v([\![g]\!])\cong\mathbb{Z}^6$. It is clear that ${\rm GCD}_{\mathbb{Z}[X]}(f,g)={\rm GCD}_{\mathbb{Z}[X]}(f,a)={\rm GCD}_{\mathbb{Z}[X]}(6,a)=\mathbb{Z}[X]^{\times}$. Along the same lines we infer that ${\rm GCD}_{\mathbb{Z}}(f(x),g(x))={\rm GCD}_{\mathbb{Z}}(f(x),a)={\rm GCD}_{\mathbb{Z}}(6,a)=\mathbb{Z}^{\times}$ for each $x\in\mathbb{Z}$. It is clear that $aX+1,aX+2,aX+3\in\mathcal{A}(\mathbb{Z}[X])$. It follows by Lemma \ref{exprep1} that $6Xf+p_i\in\mathcal{A}(\mathbb{Z}[X])$ for every $i\in [1,7]$. Consequently, it is obvious that for all $p\in\mathcal{A}(\mathbb{Z})$ and $h\in\mathcal{A}(\mathbb{Z}[X])$ with $(p\mid_{\mathbb{Z}} f(0)$ and $h\mid_{\mathbb{Z}[X]} g)$ or $(p\mid_{\mathbb{Z}} g(0)$ and $h\mid_{\mathbb{Z}[X]} f)$ it follows that $p\mid_{\mathbb{Z}[X]} h-h(0)$. Since $\frac{(aX+1)(aX+2)^2(aX+3)}{12}\in [\![f]\!]$, we infer by Proposition \ref{twoproduct} and Lemma \ref{preparation3} that $\mathcal{C}_v(H)\cong\mathcal{C}_v([\![\frac{(aX+1)(aX+2)^2(aX+3)}{12}]\!])\times\mathcal{C}_v([\![g]\!])\cong\mathbb{Z}/2\mathbb{Z}\times\mathbb{Z}^6$.
\end{proof}

Now we present a result which enables us to construct examples of monadic submonoids of ${\rm Int}(\mathbb{Z})$ whose divisor-class group is torsion-free with prescribed rank.

\begin{Pro}\label{exconcl} Let $R$ be a factorial domain, $K$ a field of quotients of $R$, $X$ an indeterminate over $K$, $\mathcal{P}$ a system of representatives of $\mathcal{A}(R)$, $n\in\mathbb{N}$, $\underline{a}\in R^n$, and $\underline{p}\in\mathcal{P}^n$ such that for all $i\in [1,n]$, $p_1\mid_R a_i-1$, $a_i+p_kR\in (R/p_kR)^{\times}$ for all $k\in [2,n]$, and if $i>1$, then $n=|\{p_j+p_1R\mid j\in [1,n]\}|=|R/p_1R|<|R/p_iR|$. Set $H=[\![\prod_{i=1}^n (a_iX-p_i)]\!]$. Then $\mathcal{C}_v(H)\cong\mathbb{Z}^{n-1}$, for every $P\in\mathfrak{X}(H)$ there is some $Q\in\mathfrak{X}(H)$ such that $(PQ)_v$ is a principal ideal of $H$, and ${\rm D}(\{[P]\mid P\in\mathfrak{X}(H)\})\geq n$.
\end{Pro}

\begin{proof} By Lemma \ref{exprep} it is sufficient to show that $d(\prod_{i=1}^n (a_iX-p_i)^{r_i})=p_1^{\min\{r_i\mid i\in [1,n]\}}$ for all $\underline{r}\in\mathbb{N}_0^n$. Let $\underline{r}\in\mathbb{N}_0^n$ and $q\in\mathcal{P}$. We need to show that $\min\{\sum_{i=1}^n r_i\mathrm{v}_q(a_ix-p_i)\mid x\in R\}=\delta_{q,p_1}\min\{r_i\mid i\in [1,n]\}$.

Case 1. $q\not=p_i$ for all $i\in [1,n]$: Observe that $\sum_{i=1}^n r_i\mathrm{v}_q(a_iq-p_i)=0$. Therefore, $\min\{\sum_{i=1}^n r_i\mathrm{v}_q(a_ix-p_i)\mid x\in R\}=0$.

Case 2. $q=p_j$ for some $j\in [2,n]$: Since $n<|R/qR|$, there is some $y\in R$ such that $q\nmid_R y-p_k$ for all $k\in [1,n]$. Let $k\in [1,n]$. Since $a_k+qR\in (R/qR)^{\times}$, there is some $z\in R$ such that $q\mid_R a_kz-y$. Consequently, $q\nmid_R a_kz-p_k$. It follows that $\sum_{i=1}^n r_i\mathrm{v}_q(a_iz-p_i)=0$. This implies that $\min\{\sum_{i=1}^n r_i\mathrm{v}_q(a_ix-p_i)\mid x\in R\}=0$.

Case 3. $q=p_1$: Let $y\in R$. Since $R/qR=\{p_i+qR\mid i\in [1,n]\}$, there is some $j\in [1,n]$ such that $q\mid_R y-p_j$. Since $q\mid_R a_j-1$, we infer that $q\mid_R a_jy-p_j$. Therefore, $\min\{r_i\mid i\in [1,n]\}\leq r_j\leq\sum_{i=1}^n r_i\mathrm{v}_q(a_iy-p_i)$, hence $\min\{r_i\mid i\in [1,n]\}\leq\min\{\sum_{i=1}^n r_i\mathrm{v}_q(a_ix-p_i)\mid x\in R\}$.

There is some $k\in [1,n]$ such that $\min\{r_i\mid i\in [1,n]\}=r_k$. We show that $\mathrm{v}_q(a_kz-p_k)=1$ for some $z\in R$. Since $q\mid_R a_k-1$, there is some $y\in R$ such that $q\mid_R a_ky-p_k$. If $q^2\nmid_R a_ky-p_k$, then set $z=y$. Now suppose that $q^2\mid_R a_ky-p_k$. Set $z=q+y$. Then $\mathrm{v}_q(a_kz-p_k)=1$.

Let $j\in [1,n]\setminus\{k\}$. Then $q\nmid_R p_j-p_k$. Since $q\mid_R a_jz-z$, we infer that $\mathrm{v}_q(a_jz-p_j)=0$. Consequently, $\min\{\sum_{i=1}^n r_i\mathrm{v}_q(a_ix-p_i)\mid x\in R\}\leq\sum_{i=1}^n r_i\mathrm{v}_q(a_iz-p_i)=r_k=\min\{r_i\mid i\in [1,n]\}$.
\end{proof}

The following result is a useful application of Corollary \ref{twoproduct2}.

\begin{Pro}\label{exconcl2} Let $R$ be a factorial domain, $K$ a field of quotients of $R$, $X$ an indeterminate over $K$, and $\mathcal{P}$ a system of representatives of $\mathcal{A}(R)$. Let $k\in\mathbb{N}$, and $(\mathcal{P}_i)_{i=1}^k$ a finite sequence of finite and pairwise disjoint subsets of $\mathcal{P}$ such that for every $i\in [1,k]$ there is some $p\in\mathcal{P}_i$ for which $|\mathcal{P}_i|=|\{r+pR\mid r\in\mathcal{P}_i\}|=|R/pR|<|R/qR|<\infty$ for all $q\in\mathcal{P}_i\setminus\{p\}$, and $p\mid_R\prod_{a\in\bigcup_{j=1,j\not=i}^k\mathcal{P}_j} a-1$. Set
\[
g=\prod_{j=1}^k\Big(\prod_{b\in\mathcal{P}_j}\Big(\Big(\prod_{a\in\bigcup_{i=1,i\not=j}^k\mathcal{P}_i} a\Big)X-b\Big)\Big).
\]
Then $\mathcal{C}_v([\![g]\!])\cong\mathbb{Z}^{\sum_{i=1}^k |\mathcal{P}_i|-k}$.
\end{Pro}

\begin{proof} For $j\in [1,k]$, set $f_j=\prod_{b\in\mathcal{P}_j} ((\prod_{a\in\bigcup_{i=1,i\not=j}^k\mathcal{P}_i} a)X-b)$. It is sufficient to show by induction on $j$ that $\mathcal{C}_v([\![\prod_{i=1}^j f_i]\!])\cong\mathbb{Z}^{\sum_{i=1}^j |\mathcal{P}_i|-j}$ for every $j\in [1,k]$. It follows immediately from Proposition \ref{exconcl} that $\mathcal{C}_v([\![f_j]\!])\cong\mathbb{Z}^{|\mathcal{P}_j|-1}$ for every $j\in [1,k]$. Let $j\in [2,k]$. Set $f=\prod_{i=1}^{j-1} f_i$ and $g^{\prime}=f_j$. Note that ${\rm GCD}_{R[X]}(f,g^{\prime})=R[X]^{\times}$. We show that ${\rm GCD}_R(f(0),g^{\prime}(0))=R^{\times}$ and for all $p\in\mathcal{A}(R)$ and $h\in\mathcal{A}(R[X])$ such that ($p\mid_R f(0)$ and $h\mid_{R[X]} g^{\prime})$ or ($p\mid_R g^{\prime}(0)$ and $h\mid_{R[X]} f)$ we have $p\mid_{R[X]} h-h(0)$. Observe that $f(0)\simeq_R\prod_{b\in\bigcup_{i=1}^{j-1}\mathcal{P}_i} b$ and $g^{\prime}(0)\simeq_R\prod_{b\in\mathcal{P}_j} b$. Since $\bigcup_{i=1}^{j-1}\mathcal{P}_i$ and $\mathcal{P}_j$ are disjoint, it follows that ${\rm GCD}_R(f(0),g^{\prime}(0))=R^{\times}$. Let $p\in\mathcal{A}(R)$ and $h\in\mathcal{A}(R[X])$.

Case 1. $p\mid_R f(0)$ and $h\mid_{R[X]} g^{\prime}$: Of course, $p\simeq_R b$ for some $b\in\bigcup_{i=1,i\not=j}^k\mathcal{P}_i$ and $h\simeq_{R[X]}\linebreak (\prod_{a\in\bigcup_{i=1,i\not=j}^k\mathcal{P}_i} a)X-c$ for some $c\in\mathcal{P}_j$. Therefore, $p\simeq_{R[X]} b\mid_{R[X]} (\prod_{a\in\bigcup_{i=1,i\not=j}^k\mathcal{P}_i} a)X\simeq_{R[X]} h-h(0)$.

Case 2. $p\mid_R g^{\prime}(0)$ and $h\mid_{R[X]} f$: It is clear that $p\simeq_R b$ for some $b\in\mathcal{P}_j$ and $h\simeq_{R[X]} (\prod_{a\in\bigcup_{i=1,i\not=l}^k\mathcal{P}_i} a)X-c$ for some $l\in [1,j-1]$ and some $c\in\mathcal{P}_l$. Consequently, $p\simeq_{R[X]} b\mid_{R[X]} (\prod_{a\in\bigcup_{i=1,i\not=l}^k\mathcal{P}_i} a)X\simeq_{R[X]} h-h(0)$.

We infer by Corollary \ref{twoproduct2} that $\mathcal{C}_v([\![fg^{\prime}]\!])\cong\mathcal{C}_v([\![f]\!])\times\mathcal{C}_v([\![g^{\prime}]\!])\cong\mathbb{Z}^{\sum_{i=1}^{j-1} |\mathcal{P}_i|-j+1}\times\mathbb{Z}^{|\mathcal{P}_j|-1}\cong\mathbb{Z}^{\sum_{i=1}^j |\mathcal{P}_i|-j}$.
\end{proof}

\begin{Ex} Let $g=(95095X+2)(95095X+3)(6X+5)(6X+7)(6X+11)(6X+13)(6X+19)\in\mathbb{Z}[X]$. Then $\mathcal{C}_v([\![g]\!])\cong\mathbb{Z}^5$.
\end{Ex}

\begin{proof} This follows from Proposition \ref{exconcl2} with $k=2$, $\mathcal{P}_1=\{2,3\}$, and $\mathcal{P}_2=\{5,7,11,13,19\}$.
\end{proof}

There are many important invariants which can describe the structure of factorizations. Two of them that are commonly used are the elasticity $\rho(H)$ and the tame degree ${\rm t}(H)$ of a monoid $H$. For the definitions of the elasticity and the tame degree we refer to \cite[Definitions 1.4.1 and 1.6.4]{GHK}. In what follows we want to provide a class of rings of integer-valued polynomials over factorial domains where both of these invariants are infinite. Note that if $H$ is an atomic monoid and $T\subseteq H$ is an atomic divisor-closed submonoid, then $\rho(T)\leq\rho(H)$ and ${\rm t}(T)\leq {\rm t}(H)$.

\begin{Th}\label{main2} Let $R$ be a factorial domain, $K$ a field of quotients of $R$, $X$ an indeterminate over $K$ and $\mathcal{P}$ a system of representatives of $\mathcal{A}(R)$. Let $(\mathcal{P}_i)_{i\in\mathbb{N}}$ be a sequence of finite subsets of $\mathcal{P}$ such that for every $i\in\mathbb{N}$ there is some $p\in\mathcal{P}_i$ for which $i<|\mathcal{P}_i|=|\{r+pR\mid r\in\mathcal{P}_i\}|=|R/pR|<|R/qR|<\infty$ for all $q\in\mathcal{P}_i\setminus\{p\}$. Then $\rho({\rm Int}(R))={\rm t}({\rm Int}(R))=\infty$.
\end{Th}

\begin{proof} For $i\in\mathbb{N}$ set $H_i=[\![\prod_{a\in\mathcal{P}_i} (X-a)]\!]$. By Proposition \ref{exconcl} we infer that $\{[P]\mid P\in\mathfrak{X}(H_i)\}=\{[P^{-1}]\mid P\in\mathfrak{X}(H_i)\}$ and ${\rm D}(\{[P]\mid P\in\mathfrak{X}(H_i)\})>i$ for all $i\in\mathbb{N}$. It follows from \cite[Theorem 3.4.10]{GHK} that $\rho({\rm Int}(R))\geq\rho(H_i)\geq {\rm D}(\{[P]\mid P\in\mathfrak{X}(H_i)\})/2>i/2$, and ${\rm t}({\rm Int}(R))\geq {\rm t}(H_i)\geq {\rm D}(\{[P]\mid P\in\mathfrak{X}(H_i)\})> i$ for every $i\in\mathbb{N}$. This implies that $\rho({\rm Int}(R))={\rm t}({\rm Int}(R))=\infty$.
\end{proof}

\bigskip
\textbf{Acknowledgement.} This work was supported by the Austrian Science Fund FWF, Project number P26036-N26.

\end{document}